\documentclass{siamart190516}
\usepackage{times}
\usepackage{frenchineq}
\usepackage{fixmath}
\usepackage{mathtools}
\usepackage{diffcoeff}
\usepackage{comment}
\usepackage[english]{babel}
\usepackage[T1]{fontenc}
\usepackage{xargs}
\usepackage{amscd}
\usepackage{amsmath}
\usepackage{amsfonts}
\usepackage{amssymb}

\usepackage{graphicx}    
\usepackage{subfigure}

\usepackage{float} 

\usepackage{algorithm}
\usepackage{algpseudocode}

\usepackage{cleveref}

\usepackage[colorinlistoftodos,bordercolor=orange,backgroundcolor=orange!20,linecolor=orange,textsize=scriptsize]{todonotes}

\usepackage{mathrsfs,epsfig,setspace,amstext,color,graphics,verbatim}

\renewcommand{\textcolor}[2]{#2} 

\theoremstyle{theorem}
\newtheorem{claim}{Claim}
\newtheorem{observation}{Observation}

\theoremstyle{example}
\theoremstyle{definition}

\newtheorem{remark}{Remark}

\newcommand{\e}{e}
\newcommand{\myt}{T_\beta}
\newcommand{\myp}{P_{\beta}}
\newcommand{\myq}{Q_{\alpha}}
\newcommand{\myqd}{Q_{\alpha,d}}

\newcommand{\origp}{{P}}
\newcommand{\myg}{g_\beta}

\newcommand{\nmyp}{P}

\newcommand{\nmyg}{g}
\newcommand{\bZ}{\mathbb{Z}}
\newcommand{\Rset}{\mathbb{R}}
\newcommand{\I}{\mathrm{I}}

\newcommand{\C}{\mathbb{C}}
\newcommand{\R}{\mathbb{R}}
\newcommand{\Z}{\mathbb{Z}}
\newcommand{\cS}{\mathcal{S}}
\newcommand{\spec}{\operatorname{spec}}
\newcommand{\Bd}{\operatorname{Bd}}
\newcommand{\cP}{\mathcal{P}}
\newcommand{\cQ}{\mathcal{Q}}
\newcommand{\cH}{\mathcal{H}}
\newcommand{\real}{\operatorname{Re}}
\newcommand{\ball}{\mathcal{B}}
\def\cA{\mathcal{A}}

\newcommand{\Top}{\operatorname{top}}
\newcommand{\diag}{\operatorname{diag}}

\let\ifX\iftrue
\title{Multiply Accelerated Value Iteration for Non-Symmetric Affine Fixed Point Problems and application to Markov Decision Processes}
\author{Marianne Akian\thanks{INRIA and CMAP, \'{E}cole polytechnique, IP Paris, CNRS. Address: CMAP, Ecole polytechnique, Route de Saclay, 91128 Palaiseau Cedex, France (emails: {marianne.akian@inria.fr}, {stephane.gaubert@inria.fr}, {omar.saadi@polytechnique.edu}).}
  \and
  St{\'e}phane Gaubert\footnotemark[2]
  \and
  Zheng Qu\thanks{Department of Mathematics, The University of Hong Kong. Address: The University of Hong Kong, Pokfulam Road, Hong Kong (email: {zhengqu@hku.hk}).}
  \and
  Omar Saadi\footnotemark[2]
\thanks{O. Saadi acknowledges the support of the Hassan II Academy of Science and Technology. The authors acknowledge the support of the Gaspard Monge (PGMO) program of Fondation Math\'ematique
  Hadamard, EDF, Orange and Thales, of the ICODE institute of Paris-Saclay,
  and of the ``Investissement d'avenir'' r\'ef\'erence ANR-11-LABX-0056-LMH, LabEx LMH. Z. Qu acknowledges the support of Hong Kong Research Grants Council No. 27302016.}}

\begin{document}

\maketitle


\begin{abstract}                
  We analyze a modified version of Nesterov accelerated gradient algorithm,
  which applies to affine fixed point problems with non self-adjoint
  matrices, such as the ones appearing in the theory of Markov decision
  processes with discounted or mean payoff criteria.
We characterize the spectra of matrices for which this algorithm
does converge with an accelerated asymptotic rate.
  We also introduce a $d$th-order
  algorithm, and show that it yields a multiply accelerated rate
  under more demanding conditions on the spectrum. We subsequently
  apply these methods to develop accelerated schemes for non-linear
  fixed point problems arising from Markov decision processes.
  This is illustrated by numerical experiments. 
\end{abstract}

%

\section{Introduction}
The dynamic programming method reduces
optimal control and repeated zero-sum game problems
to fixed
point problems involving non-linear operators that are
order preserving and sup-norm nonexpansive, see~\cite{Bellman,puterman2014markov}
for background. The
0-player case, with a finite number $n$ of states, 
is already of interest.
\textcolor{blue}{In this case, the involved operator is $T:\R^n \to \R^n$ of the form $T(x)= g+Px$, 
where $g=(g_i)\in \R^n$ and $P=(P_{ij})\in \R^{n\times n}$ is a substochastic matrix, i.e.\ a matrix with nonnegative entries such that the sum of each row is less than or equal to $1$.}
The scalar $g_i$ is
an instantaneous payment received in state $i$, whereas
$P_{ij}$ is the transition probability from $i$ to $j$. The
difference $1-\sum_j P_{ij}$ is the probability
that the process terminates, when in state $i$.
If $v$ is a fixed point of $T$, the entry $v_i$ yields
the expected cost-to-go from the initial state $i$.
More generally, in the one player case (Markov decision processes),
one needs to solve a non-linear fixed point problem,
described in Section~\ref{sec-numerical}, in which the operator
$T$ is now a supremum of affine operators $x\mapsto g+Px$.

The standard method to obtain the fixed point of $T$
is to compute the sequence $x_k=T(x_{k-1})$, this is known as {\em value
  iteration}~\cite{Bellman}.
In the $0$-player case, value iteration has an asymptotic (geometric)
convergence rate given by the spectral radius of $P$.
In many applications, this spectral radius is of the form $1-\epsilon$
where $\epsilon$ is {\em small}. E.g.,
$\epsilon$ may represent a discount rate.
We look for {\em accelerated}
fixed point algorithms, with a convergence rate $1-\Omega(\epsilon^{1/d})$
for some $d\geq 2$, \textcolor{blue}{i.e.\ a convergence rate that is smaller than $1- c \epsilon^{1/d}$ for some constant $c>0$.}

In the special case of $0$-player problems with a {\em  symmetric}
matrix $P$, an algorithm with a rate  $1-\Omega(\epsilon^{1/2})$
can be obtained by specializing the accelerated gradient algorithm
of Nesterov~\cite{nesterov83}.
The latter algorithm applies to the minimization of a smooth strictly convex function $f$,
which, in the quadratic case, reduces to an affine fixed point problem
with a symmetric matrix $P$. See~\cite{Flammarion15}.
In contrast, developing accelerated algorithms
for problems of non-symmetric type is a challenging question,
which has been studied recently in~\cite{iutzeler,julien}.

 We study here the affine fixed point
problem $x=g+Px$ where the matrix $P$ is non symmetric,
and possibly not substochastic. \Cref{ThmSigma}, one of our main results, states that
a modification of Nesterov's scheme~\cite{nesterov83} does converge
with an asymptotic rate $1-\epsilon^{1/2}$
if the spectrum of $P$ is contained in an explicit
region of the complex plane, obtained
as the image of the disk of radius $1-\epsilon$ by a rational
function of degree $2$. 
We also show that the
 incorporation of a
Krasnosel'ski\u\i-Mann type damping~\cite{mann,krasno} (see~\Cref{a:yk})  enlarges the admissible spectrum region of $P$ for acceleration, see~\Cref{th:flyingsaucer}.
Moreover, we introduce a new scheme~\eqref{dAVI},
of order $d\geq 2$, and show in~\Cref{dSigmaAcc} that it leads to a multiply accelerated asymptotic
rate of $1-\epsilon^{1/d}$, but under a
more demanding condition on the spectrum of $P$, see~\Cref{dSigmaAcc}.
This theorem also shows that this condition is tight.
However, slightly more flexible conditions suffice
to guarantee a rate of $1-\Omega(\epsilon^{1/d})$,
as shown by~\Cref{thm:daccle}.

We subsequently apply the proposed schemes and theoretical results, concerning the affine ``$0$-player case'' , to solve
 non-linear
fixed point problems arising from Markov decision processes.
We use policy iteration, which allows a reduction
to a sequence of affine fixed point problems,
still benefiting of acceleration for the solution of each
affine problem.
This leads to an accelerated policy iteration algorithm (see~\Cref{d-API}),
\textcolor{blue}{which produces an approximate solution with a precision 
of order $((1+\gamma)\delta+\delta')/(1-\gamma)^2$
where $\gamma$ is the maximal discount factor}, $\delta$ is the accuracy of each inner affine problem and $\delta'$ is the accuracy of the policy improvement,
see~\Cref{prop-dAPI}.

In~\Cref{sec-numerical}, we show the performance of the simple and multiple acceleration schemes, on classes of instances in which the spectral conditions
for acceleration are met. 
In~\Cref{subsec-random}, we consider a framework of random matrices that shows distributions of eigenvalues~\cite{Bordenave08} that are compatible with the spectral conditions required for the convergence of the simple and multiple acceleration schemes proposed here. In~\Cref{subsec-HJB}, we show the performance of the accelerated schemes in solving a Hamilton-Jacobi-Bellman equation in the case of small drifts. This example illustrates the usefulness of~\Cref{ThmSigmastronger} that allows to have a more tolerant accelerable region on the complex plane while still benefiting from an accelerated asymptotic rate of $1-\Omega(\epsilon^{1/2})$. 


The recent works \cite{iutzeler,julien} also deal with generalizations of Nesterov's accelerated algorithm to solve fixed point problems.
Their theoretical convergence results apply to matrices with
a real spectrum, showing that the original choice of parameters for
Nesterov's method in the symmetric case still yields an
acceleration in this setting. In contrast, we allow a complex spectrum
and characterize the region of the complex plane containing
spectra of matrices for which the acceleration is valid (see~\Cref{ThmSigma} and~\Cref{th:flyingsaucer}).
Also, a main novelty of the present work
is the analysis of multiple accelerations~\eqref{dAVI}.
The idea of applying Nesterov's acceleration to Markov decision
processes appeared in~\cite{julien}, in which a considerable
experimental speedup is reported on random instances.
The algorithm there coincides with one of the algorithms
studied here --  $2$-accelerated value iteration for Markov decision
processes. It is an open problem 
to establish the convergence of this method
for large enough classes of Markov decision processes. The characterization
of the set of ``accelerable'' 0-player problems that we
provide here explains why this problem is inherently difficult:
in the $0$-player problem, the convergence conditions
are governed by fine spectral properties which have
no known non-linear analogue in the one-player case.

Apart from being applied to Markov decision processes,  fixed point iteration 
also includes as a special case the proximal point method~\cite{Rockafellar76}, when the mapping $T$  corresponds to the resolvent of a maximal monotone operator.  The proximal point method 
covers  a list of pivotal  algorithms in optimization  such as the proximal gradient descent, the augmented Lagrangian method (ALM)~\cite{Rockafellar1976Augmented} and the alternating directional method of multipliers (ADMM)~\cite{Eckstein1992On}. 
The development of accelerated proximal point method has thus attracted a lot of attention~\cite{ChenMaYang,AttouchPeypouquet2020,Attouch2020} and a recent paper~\cite{kim2019accelerated}  constructed a new algorithm achieving  $\|x_k-T(x_k)\|\leq O(1/k)$ through the performance estimation problem (PEP) approach~\cite{DroriTeboulle}. 
In a more general setting when $T$  is a nonexpansive mapping in a Euclidean norm, a version of Halpern's iteration was recently shown
to yield a residual $\|x_k-T(x_k)\|\leq O(1/k)$~\cite{Lieder}, also via the PEP approach. These results improve over the worst case bound 
$\|x_k-T(x_k)\|\leq O(1/\sqrt{k})$ of   the  Krasnoselski-Mann's iteration for a nonexpansive mapping (in arbitrary norm)~\cite{baillon1996bruck}.
The acceleration results in the above cited works
do not overlap with ours as they only 
apply to nonexpansive mappings in a Euclidean norm.  Moreover, in this paper we consider
strictly contractive mapping and thus focus on linear instead of sublinear convergence guarantees. 

There is also a large body of literature on (quasi-)Newton type methods for solving nonlinear equations~\cite{Fang2009TwoCO,IzmailoSolodov,WalkerNi}, which can
be naturally employed for solving fixed point problem and yield fast asymptotic convergence rate. 
It is well-known that such methods converge only when close enough to the solution. Some  papers
proposed various safe-guard conditions to globalize the convergence~\cite{superman,zhang2018globally} and do not provide a rate
of convergence.  We formally characterize the spectrum condition and the faster convergence rate 
of accelerated value iteration for affine fixed
point problem.

The paper is organized as follows. In~\Cref{sec:AVI} we introduce the accelerated value iteration (AVI) of any degree $d\geq 2$.  In~\Cref{sec:d2} we provide a formal analysis of AVI of degree 2. In~\Cref{sec:d3} we analyze AVI of arbitrary 
degree $d\geq 2$ and also present the application to Markov decision processes. In~\Cref{sec-numerical} we provide numerical experimental results.




\nocite{krasno,mann}


\section{Accelerated Value Iteration}\label{sec:AVI}






Nesterov proposed in~\cite{nesterov83,Nesterovbook} to accelerate the gradient descent scheme for the minimization of a  $\mu$-strongly convex function $f:\R^n\rightarrow \R$ whose gradient is of Lipschitz constant $L$, by adding an inertial step:
\begin{subequations}
  \label{a:agd}
\begin{align}
x_{k+1}&=y_k - h \nabla f(y_k) \enspace,\\
y_{k+1}&=x_{k+1}+\alpha(x_{k+1}-x_{k})\enspace,
\end{align}
\end{subequations}
where $0<h$, and $\alpha \in [0,1]$ are parameters.
Let $x_*$ be the minimizer of $f$.
When $\alpha=0$,~\eqref{a:agd} reduces to gradient descent.
With the step $h=1/L$, the gradient descent
converges linearly with a rate $1-2\mu/(L+\mu)$.
Indeed, we have $\|x_k-x_*\|^2
\leq \big(1-2\mu/(L+\mu) \big)^k \|x_0-x_*\|^2$,
and $f(x_k)- f(x_*) \leq \frac{L}{2}
 \big(1-2\mu/(L+\mu)\big)^k \|x_0-x_*\|^{2}$,
for all $k\geq 1$, see
Theorem 2.1.14 in \cite{Nesterovbook}.
Moreover, Theorem 2.2.3, {\em ibid.}, implies 
that if we choose
\begin{align}\label{a:alpha}
\alpha=\frac{1-\sqrt{\mu/L}}{1+\sqrt{\mu/L}},
\end{align}
still with $h=1/L$, the scheme~\eqref{a:agd} converges linearly with a rate $1-\sqrt{\mu/L}$. Indeed, with $\alpha$ given by~\eqref{a:alpha}, we have $f(x_k)-f(x_*) \leq 2(1-\sqrt{\mu/L})^k(f(x_0)-f(x_*))$ for all $k\geq 1$. Note that when the condition number $L/\mu$ is large, i.e. $L/\mu \gg 1$, the rate  $1-\sqrt{\mu/L}$ improves over   $1-2\mu/(L+\mu)$, whence the scheme~\eqref{a:alpha} is commonly known as \textit{accelerated gradient descent}.

We consider the fixed point problem for the operator
 \begin{align}
\label{a:T}T(x) =g+\nmyp x \enspace.
\end{align}
\textcolor{blue}{Here, we allow the vector $x$ and the matrix $\nmyp$ to have complex entries, requiring only the spectral radius of the matrix $\nmyp$ to be strictly less than $1$. In the application to MDPs, the vector $x$ will be real and the matrix $P$ will be nonnegative.
}
By abuse of notation, we  denote by $x_*$ the unique fixed point of $T$.  
We study the {\em Accelerated Value Iteration algorithm (AVI)} 
for computing a fixed point of the operator $T$. It makes
a Krasnosel'ski\u\i-Mann (KM) type damping of parameter $0<\beta\leq 1$,
replacing $T$ by $(1-\beta)I + \beta T$,
followed by a Nesterov acceleration step:
\begin{subequations}
  \label{AVIalphabeta}
\begin{align}
  x_{k+1}&=
  (1-\beta) y_k + \beta T(y_k) \enspace ,\label{a:yk} \\ 
y_{k+1}&=x_{k+1}+\alpha(x_{k+1}-x_{k})
\enspace .
\end{align}
\end{subequations}
When $\alpha=0$ and $\beta=1$, the scheme~\eqref{AVIalphabeta} reduces to the standard fixed point iteration algorithm:
\begin{equation}\label{e-VI}
x_{k+1}=\nmyg+\nmyp x_k \enspace.
\end{equation}
When the spectral radius of $P$ is smaller than $1-\epsilon$ for some $\epsilon \in (0,1)$, the standard fixed point scheme converges with an asymptotic rate no greater than $1-\epsilon$ to the unique fixed point, meaning that \textcolor{blue}{for any norm $\| \cdot \|$}
$$\limsup_{k\to\infty}\|x_k-x_*\|^{1/k}
\leq 1-\epsilon.$$

By analogy with accelerated gradient descent, we aim at accelerating the standard fixed point scheme by finding appropriate parameters $\alpha$ and $\beta$ so that 
 \begin{align}\label{a:sqrt}\limsup_{k\to\infty}\|x_k-x_*\|^{1/k}
\leq 1-\sqrt{\epsilon},\end{align}
for matrices $P$ with spectral radius bounded by $1-\epsilon$.
\begin{remark}\rm\label{rem:quadraticf}
  If $P$ is symmetric, the iteration~\eqref{AVIalphabeta} can be recovered
  by applying the accelerated gradient descent 
scheme~\eqref{a:agd} to the quadratic function $f(x)\equiv \frac{1}{2} x^\top (I-P)x-g^\top x$. The damping parameter $\beta$ corresponds to the step $h$.
  However, Nesterov's results only apply to the case when $f$ is a strongly convex function. This requires in particular $I-P$ to be symmetric
  positive definite. In particular all the eigenvalues of $P$ must be real
  and smaller than $1$. 
\end{remark}
The scheme~\eqref{AVIalphabeta} for fixed point iteration 
 has been considered recently by
\cite{iutzeler,julien}. Moreover, inspired by the momentum method~\cite{POLYAK19641,7330562} for improving gradient descent, 
~\cite{julien} also proposed a momentum fixed point method described as follows:
	\begin{align}
	\label{mFP}
	x_{k+1}&=(1-\beta)x_k+ \beta T(x_k)+\alpha (x_k-x_{k-1}).
\end{align}
Asymptotic rate analysis for~\eqref{AVIalphabeta}~\eqref{mFP} follows from~\cite{julien} when the spectrum of $P$ is real. 

As discussed in the introduction,
our main results apply to complex spectra,
and also to higher degree of acceleration.

In the scheme~\eqref{AVIalphabeta}, $y_{k+1}$ is generated from a linear combination of the last two iterates. We now consider  the following {\em Accelerated Value Iteration of degree $d$ ($d$A-VI)}, in which  $y_{k+1}$ is   a linear combination of the last $d$ iterates for any $d\geq 2$,
  \begin{subequations}
	\label{dAVI}
	\begin{align}
	x_{k+1}&=(1-\beta)y_k+ \beta T(y_k) \enspace ,\label{dKM}\\
	y_{k+1}&=(1+\alpha_{d-2}+\cdots+\alpha_0)x_{k+1} -\alpha_{d-2} x_k -\cdots -\alpha_0 x_{k-d+2} \enspace . \label{dAVIy}
	\end{align}
\end{subequations}
 We will show how to select the parameters $\alpha=(\alpha_0, \cdots, \alpha_{d-2})$
to obtain an acceleration of  order $d$, in the sense that 
\begin{align}\label{a:dferef}\limsup_{k\to\infty}\|x_k-x_*\|^{1/k} 
\leq 1-{\epsilon}^{1/d}.\end{align}

\textcolor{blue}{
\begin{remark}\rm\label{rem:AA}
  The idea of accelerating the vanilla KM fixed point method by extrapolating a finite number of previous steps
  goes back to the work of Anderson in 1965~\cite{Anderson}. The algorithm known as Anderson Acceleration
  (AA) chooses dynamically the extrapolation coefficients, while the coefficients $\alpha=(\alpha_0, \cdots, \alpha_{d-2})$ in $d$A-VI~\eqref{dAVI} remain constant for all the iterations. The theoretical analysis of AA and of its variants is still under development. In particular, the theoretical convergence rate of AA seems to be missing in the literature, except in the special case when $T$ corresponds to the gradient descent mapping of a strongly convex and smooth function~\cite{SBA17}. When $T$ takes the form of~\eqref{a:T}, this requires $P$ to be symmetric, see~\Cref{rem:quadraticf}. In~\cite{zhang2018globally}, a modified AA, interleaving KM updates by using safe-guarding steps, is shown to be globally converging, but the convergence rate is not analyzed. As shown later, the $d$-AVI~\eqref{dAVI} does not need any safe-guard checking and will converge with 
accelerated asymptotic rate as in~\eqref{a:dferef} under some conditions on the spectrum of $P$.
\end{remark}}
\textcolor{blue}{
\begin{remark}\rm\label{rem:computational-cost}
The computational cost of one iteration of the classical Value Iteration algorithm~\eqref{e-VI} is $O(n^2)$. In comparison, the computational cost of one iteration of the $d$A-VI algorithm~\eqref{dAVI} is $O(n(n+d))$. Regarding the space complexity, the classical Value Iteration needs to store two vectors ($x_{k+1},x_k$) each of size $n$, so it needs a $2n$ space of memory. In comparison, the $d$A-VI algorithm needs to store $d+1$ vectors ($y_{k+1},x_{k+1},\cdots,x_{k-d+2}$) each of size $n$, so it needs $(d+1)n$ space of memory. We notice that in practice the degree $d$ that we will use is small ($\leq 4$), therefore the computational cost of one iteration of $d$A-VI and its space complexity are similar to the ones of the classical Value Iteration algorithm. Moreover, the asymptotic convergence rate $1-\epsilon^{1/d}$ allows the $d$A-VI algorithm to converge in a number of iterations smaller than the Value Iteration algorithm (see the numerical experiments in~\Cref{sec-numerical}).
\end{remark}
}





\if{\subsection{The effect of the Krasnosel'ski\u\i-Mann scaling}
The use of the Krasnosel'ski\u\i-Mann iteration is equivalent to looking for the fixed point of the scaled operator $\myt:x\mapsto \myg +\myp x$ instead of $T$.
So for all $x \in \Rset^{n}$, $\myt(x)=(1-\beta)x+\beta T(x)=\myg+\myp x$, with $\myg=\beta g$, and $\myp=(1-\beta)\I+\beta \origp$, where $\I$ is the identity matrix. We denote by $\spec\origp$ the spectrum of a matrix $\origp$.
\todo{is it the good place to define $\spec$?}
\begin{observation}\label{RemHomothety}
  The spectrum of $\myp$ is the image in the complex plane of the spectrum of $\origp$ by the homothety $H^{\beta}:=z\mapsto 1+ \beta (z-1)$ of center $1$ and ratio $\beta$.
\end{observation}
In the following we will show that if the spectrum of $\myp$ belongs to a certain region of the complex plane (the {\em accelerable region}) then the AVI algorithm does converge. Therefore, thanks to Observation~\ref{RemHomothety}, by choosing an appropriate $\beta \in ]0,1]$, AVI does converge if the spectrum of the initial matrix $\origp$ is included in any image of the accelerable region by an homothety of center $1$ and with any given ratio $\beta'=\frac{1}{\beta}\in [1,+\infty[$. So, in the following we will restrict the study of the accelerable region to the spectrum of $\myp$.

        The AVI algorithm can be rewritten as follows:

\begin{equation}\label{AVI}
\left\{ \begin{array}{ll} x_{k+1}=\myg+\myp y_k \\
y_{k+1}=x_{k+1}+\alpha(x_{k+1}-x_{k})
\end{array} \right., \; k=0,1,\cdots .
\end{equation}}\fi

    \section{Analysis of Accelerated Value Iteration of degree $2$}\label{sec:d2}
In this section we analyse the AVI scheme~\eqref{AVIalphabeta}.
We will show that with an appropriate choice of the acceleration parameter $\alpha$, and under an assumption on the shape of the complex spectrum of $\origp$, the asymptotic rate can  indeed be improved up to $1-\sqrt{\epsilon}$. We also show that the damping parameter $\beta$
will allow us
to enlarge the convergence region, while keeping
the acceleration properties. We deal separately with the special $d=2$ case,
since it is more elementary, easier to compare with existing
acceleration schemes, and since it gives insight on the generalization
to the higher degree case which will be done in~\Cref{sec:d3}. 

\subsection{The spectrum of the AVI iteration}
We define $\myp:= (1-\beta) I + \beta \nmyp$. 
Then, the AVI algorithm~\eqref{AVIalphabeta} can be written as the second order iteration 
\begin{align}\label{a:avi}y_{k+1}=\beta \nmyg +(1+\alpha)\myp y_{k}-\alpha \myp y_{k-1}.\end{align} Considering $z_k=y_k-x_*$, the iteration becomes $z_{k+1}=(1+\alpha)\myp z_{k}-\alpha \myp z_{k-1}$.
This is equivalent to:
\begin{equation}\label{AVIQ}
\begin{pmatrix} 
z_{k+1} \\
z_k
\end{pmatrix}=
\begin{pmatrix} 
(1+\alpha)\myp & -\alpha \myp \\
I & 0 
\end{pmatrix}
\begin{pmatrix} 
z_{k} \\
z_{k-1}
\end{pmatrix} \enspace 
\end{equation}

Without loss of generality we  first deal with the case with no damping, i.e., $\beta=1$.  The discussion for  general $\beta\in (0,1]$ can be found in~\Cref{sec:thickness} . Then, the matrix appearing
in~\eqref{AVIQ} becomes
\begin{equation}\label{a:myq}
\myq:=
\begin{pmatrix} 
(1+\alpha)\nmyp & -\alpha \nmyp \\
I & 0 
\end{pmatrix}
\enspace. 
\end{equation}
\textcolor{blue}{The asymptotic rate of convergence of the sequence $(z_k)$ in the system~\eqref{AVIQ}, when it is converging, and thus of the sequence $(y_k)$ in the AVI scheme~\eqref{AVIalphabeta} is determined by the spectral radius of $\myq$.} Recall that we want to improve this asymptotic rate
, thus it suffices to find appropriate values of $\alpha$ such that the spectral radius of $\myq$ is as small as possible.

We first relate the eigenvalues of
$\myq$ with those of $\nmyp$.
We introduce the following rational function of degree $2$, defined on $\C\setminus\{\alpha/(1+\alpha)\}$ by 
$$\phi_\alpha(z):= \frac{z^2}{(1+\alpha) z - \alpha } \enspace. 
$$
The following is a standard property of block-companion matrices,
we provide the proof for completeness.
\begin{lemma}\label{SpecQ}
	If $\alpha\neq 0$ then $\lambda$ is an eigenvalue of $\myq$ if and only if there exists an eigenvalue $\delta$ of $\nmyp$ such that $\delta=\phi_\alpha(\lambda)$. 
In other words, $$\spec \myq = \phi_\alpha^{-1}(\spec \nmyp).$$
\end{lemma}
\ifX
\begin{proof}
  Let $\lambda$ be an eigenvalue of $\myq$.
  There exists a non-zero vector $\left(\begin{smallmatrix} z_1 \\	z_0
  \end{smallmatrix}\right) \in \Rset^{2n} $ such that $\myq \left(\begin{smallmatrix}  z_1 \\	z_0   \end{smallmatrix}\right) = \lambda
  \left(\begin{smallmatrix} z_1 \\	z_0 \end{smallmatrix}\right)$.
  This is equivalent to
  $	(1+\alpha)\nmyp z_1-\alpha \nmyp z_0=\lambda z_1$
  and $z_1=\lambda z_0 $, or equivalently
  $
  (\lambda(1+\alpha)-\alpha) \nmyp z_0=\lambda^2 z_0$
  and
$	z_1=\lambda z_0 $.
  We have $z_0\neq 0$, because otherwise $z_1=\lambda z_0=0$. We notice that
  $\lambda(1+\alpha)-\alpha \neq 0$, because otherwise $\lambda=\frac{\alpha}{1+\alpha}\neq 0$ and $\lambda^2 z_0=0$, then $z_0=0$, which is not true. Therefore $\nmyp z_0=\phi_\alpha(\lambda) z_0$ which allows to conclude.
\end{proof}
\fi

\subsubsection{The case of real eigenvalues}

We now explain how to select $\alpha$ optimally. 
We first suppose that the spectrum of $\nmyp$ is real and nonnegative, i.e., $\spec\nmyp \subset [0,1-\epsilon]$ for some $\epsilon \in (0,1)$. 
  We denote by $\ball(z,r)$ the closed disk of the complex plane with center $z$ and
  radius $r$.
  We consider the minimax problem
  \begin{align}
    \min_{ \alpha>0}\;  \max_{P :\, \spec P\subset [0,1-\epsilon]} \rho(Q_\alpha)
    \label{e-minimax}
  \end{align}
  where $\rho$ denotes the spectral radius,
  and the matrix $Q_\alpha$, depending on $P$, is defined
  by~\eqref{a:myq}.

\begin{lemma}\label{lem-optalpha}
  The solution $\alpha^*$ of the minimax problem~\eqref{e-minimax}
  is given by
  \begin{align}\label{a:alphastar}\alpha^*=\frac{1-\sqrt{\epsilon}}{1+\sqrt{\epsilon}}.\end{align}
  It guarantees that $\spec Q_{\alpha^*} \subset \ball(0,1-\sqrt{\epsilon})$,
  for all matrices $P$ such that $\spec\nmyp \subset [0,1-\epsilon]$.
\end{lemma}
\begin{proof}
	By~\Cref{SpecQ},  $\lambda \in \spec\myq$ if and only if there exists $\delta \in \spec\nmyp\subset [0,1-\epsilon]$, such that $\delta=\frac{\lambda^2}{(1+\alpha)\lambda-\alpha}$. This can be written as a second degree equation in $\lambda$:
	\begin{align}\label{lambdaEq}
	\lambda^2-(1+\alpha)\delta \lambda+\alpha \delta=0.
	\end{align}
	The discriminant of this equation is $\Delta=\delta^2(1+\alpha)^2-4\alpha \delta=\delta(1+\alpha)^2(\delta-\alpha')$, where $\alpha':=\frac{4 \alpha}{(1+\alpha)^2}$.  We note that the function $\alpha\mapsto \frac{4 \alpha}{(1+\alpha)^2}$ is a strictly increasing bijection from $[0,1]$ to itself, with inverse function $\alpha\mapsto \frac{1-\sqrt{1-\alpha}}{1+\sqrt{1-\alpha}}$. Hence  $\alpha'\geq 1-\epsilon$ if and only if $\alpha \geq \alpha^*$.
	\begin{claim}\label{claim1}
		For fixed $\alpha$, the maximal modulus of the solutions of
~\eqref{lambdaEq} is increasing with $\delta$.
	\end{claim}
	\begin{proof}[Proof of Claim~\ref{claim1}]
		If $\Delta\leq 0$, i.e.\ $\delta\leq \alpha'$, then the solutions of (\ref{lambdaEq}) are complex conjugate $\lambda_{\pm}=\frac{1}{2}(\delta(1+\alpha)\pm i \sqrt{\delta(1+\alpha)^2(\alpha'-\delta)})$, and we have $\lambda_{+} \lambda_{-}=\alpha \delta$. Then $|\lambda_+|=|\lambda_-|=\sqrt{\alpha \delta}$, which is increasing in $\delta$.
		
		If $\Delta\geq 0$, i.e.\ $\delta\geq \alpha'$, the solutions of (\ref{lambdaEq}) are real: $$\lambda_{\pm}=\frac{1}{2}(\delta(1+\alpha)\pm\sqrt{\delta(1+\alpha)^2(\delta-\alpha')}).$$ Then $\max\left(|\lambda_{+}|, |\lambda_{-}|\right)=\frac{1}{2}(\delta(1+\alpha)+\sqrt{\delta(1+\alpha)^2(\delta-\alpha')})$ is strictly increasing in $\delta$.
	\end{proof}
	
	\Cref{claim1} shows that
	 \begin{align}\label{a:maxPdelta}
    \max_{P :\, \spec P\subset [0,1-\epsilon]} \rho(Q_\alpha) =\max\left\{ |\lambda|:  \lambda^2-(1+\alpha)(1-\epsilon) \lambda+\alpha (1-\epsilon)=0.\right\}
  \end{align}
The discriminant of the second order equation in~\eqref{a:maxPdelta} is $$\Delta=(1-\epsilon)^2(1+\alpha)^2-4\alpha (1-\epsilon)=(1-\epsilon)(1+\alpha)^2(1-\epsilon-\alpha').$$
	If $\alpha \geq \alpha^*$, then $\alpha'\geq  1-\epsilon$ and $\Delta\leq 0$.   In this case  $|\lambda_+|=|\lambda_-|=\sqrt{\alpha (1-\epsilon)}$ is increasing in $\alpha \in \left[\frac{1-\sqrt{\epsilon}}{1+\sqrt{\epsilon}},1\right]$. If $\alpha \leq \alpha^* $, then $\alpha'\leq  1-\epsilon$ and $\Delta\geq 0$. 
In this case $$\max\left(|\lambda_+|, |\lambda_-|\right)=\frac{1}{2}\left((1-\epsilon)(1+\alpha)+\sqrt{(1-\epsilon)^2(1+\alpha)^2-4\alpha (1-\epsilon)}\right).$$
	\begin{claim}\label{fdecreasing} %
		The function $\digamma:\alpha \rightarrow  (1-\epsilon)(1+\alpha)+\sqrt{(1-\epsilon)^2(1+\alpha)^2-4\alpha (1-\epsilon)}$ is strictly decreasing on $\left[0,\frac{1-\sqrt{\epsilon}}{1+\sqrt{\epsilon}}\right ]$.
	\end{claim}
	\begin{proof}[Proof of Claim~\ref{fdecreasing}]
		We have $$\digamma'(\alpha)=1-\epsilon+\frac{2(1-\epsilon)^2(1+\alpha)-4(1-\epsilon)}{2\sqrt{(1-\epsilon)^2(1+\alpha)^2-4\alpha (1-\epsilon)}}=\frac{ (1-\epsilon)h(\alpha
			)}{\sqrt{(1-\epsilon)^2(1+\alpha)^2-4\alpha (1-\epsilon)}},$$
		where $h(\alpha)=\sqrt{(1-\epsilon)^2(1+\alpha)^2-4\alpha ( 1-\epsilon)}+(1-\epsilon)(1+\alpha)-2.$ It is easy to check that
		$h(\alpha)=\sqrt{(2-(1-\epsilon)(1+\alpha))^2-4\epsilon}-(2-(1-\epsilon)(1+\alpha))<0$ for all $\alpha \in [0,\frac{1-\sqrt{\epsilon}}{1+\sqrt{\epsilon}}]$.  Since $2-(1-\epsilon)(1+\alpha)\geq 0$ for all $\alpha \in \left[0,\frac{1-\sqrt{\epsilon}}{1+\sqrt{\epsilon}}\right]$, we deduce that $h(\alpha) <0$ and hence $\digamma'(\alpha)<0$ for all $\alpha \in \left[0,\frac{1-\sqrt{\epsilon}}{1+\sqrt{\epsilon}}\right]$.  
	\end{proof}
	We conclude that the best choice of $\alpha$ which minimizes the
        maximum of the spectral radius of $Q_{\alpha}$ corresponding to all $P$
        with  spectrum in $[0,1-\epsilon]$ is $\alpha^*$ given in~\eqref{a:alphastar}, and it allows to have $\spec Q_{\alpha^*} \subset \ball(0,1-\sqrt{\epsilon})$ for all such matrix $P$. 
\end{proof}

\begin{remark}\rm
If $P$ is symmetric, then
the quadratic function $f$ in Remark~\ref{rem:quadraticf} is a strongly convex function with $L=1$ and $\mu=\epsilon$. 
In this special case the $\alpha^*$ in Lemma~\ref{lem-optalpha} coincides with the inertial parameter~\eqref{a:alpha} in Nesterov's constant-step method.
The same choice of step has been proposed, for nonsymmetric matrices with real
spectrum, in~\cite{iutzeler}.
\end{remark}


\subsubsection{The case of complex eigenvalues}
Now we do not assume any more that $\nmyp$ has a real spectrum.
We will show that
the best acceleration rate achievable in the case of a real spectrum, obtained by choosing $\alpha=\alpha^*$ as in Lemma~\ref{lem-optalpha},
is still achievable in the case of a complex spectrum satisfying a geometric condition.

Consider the following simple closed curve $\Gamma_{\epsilon}$  defined by the parametric equation:
	\[
	\theta \mapsto \frac{(1-\epsilon)\e^{2i \theta}}{2 \e^{i \theta}-1} \quad, \quad \theta \in (0,2 \pi] \enspace.
	\]
	Denote by $\Sigma_{\epsilon}$ the compact set delimited by the curve $\Gamma_{\epsilon}$. We show in Figure~\ref{fig:sigma0} the curve $\Gamma_0$ and the enclosed region $\Sigma_0$. It is easy to see that 
	$\Gamma_{\epsilon}$ (resp. $\Sigma_{\epsilon}$) is a scaling of $\Gamma_0$ (resp. $\Sigma_{0}$) by $1-\epsilon$.  
	Moreover, we have
	\begin{align}\label{a:Sig}\left |\frac{e^{2i \theta}}{2 e^{i \theta}-1}\right |=\frac{1}{|2 e^{i \theta}-1|}\leq \frac{1}{|2 e^{i \theta}|-1}=1, \end{align}
	and thus
 the curve $\Gamma_{\epsilon}$ is included in the  disk  $\ball(0,1-\epsilon)$. 
	It follows that  \begin{align}\label{a:sigincl}\Sigma_{\epsilon} \subset \ball(0,1-\epsilon).\end{align}

	\begin{figure}[htbp]
	\label{fig:sigma0}
	\centering
	\subfigure[]{
	\label{sigma0a}
\includegraphics[width=0.45\textwidth]{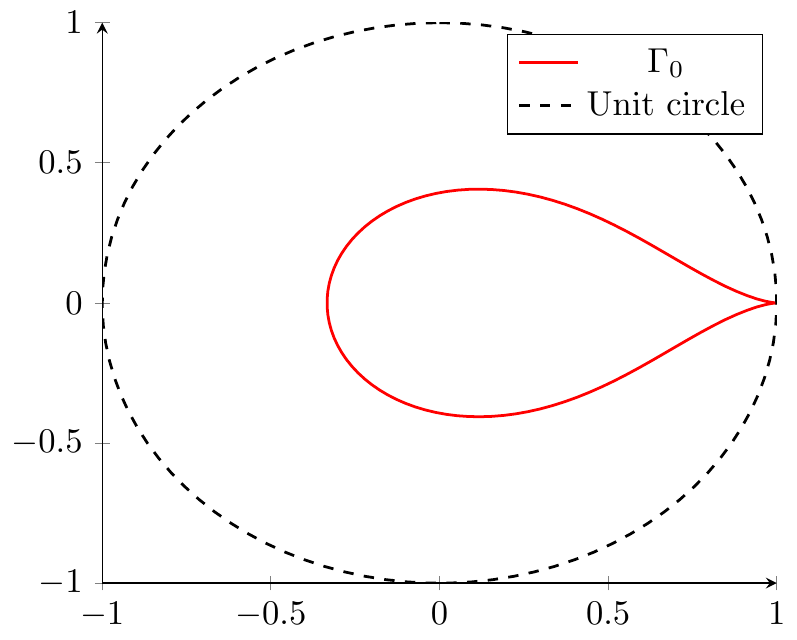}
}
\subfigure[]{
\label{sigma0b}
\includegraphics[width=0.45\textwidth]{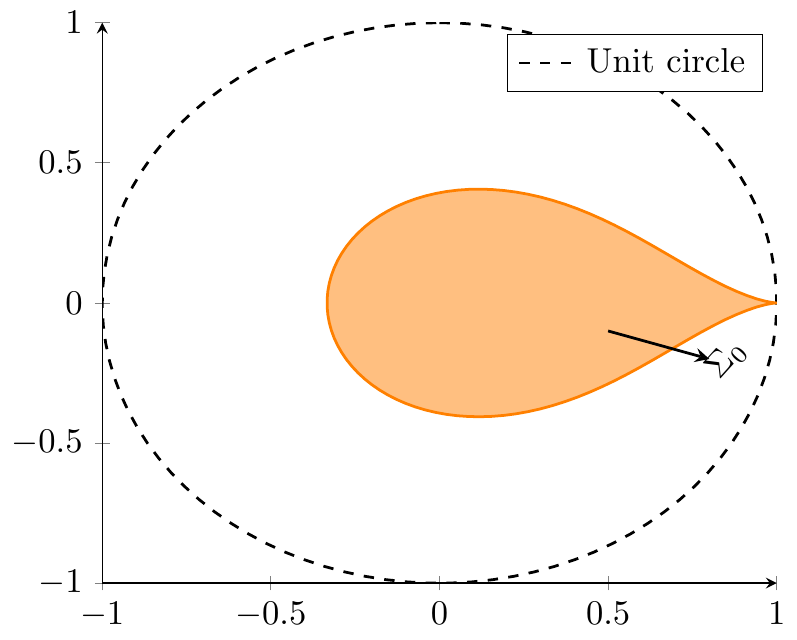}
}
\caption{Illustration of the curve $\Gamma_0$ (\Cref{sigma0a}) and its enclosed region $\Sigma_0$ (\Cref{sigma0b}).}
\end{figure}


\begin{theorem}\label{ThmSigma}
\textcolor{blue}{Let $\epsilon \in (0,1)$, $\nmyp$ be a $n\times n$ complex matrix and $\myq$ be defined as in~\eqref{a:myq} with}
$\alpha=(1-\sqrt{\epsilon})/(1+\sqrt{\epsilon})$
  . 
  If $\spec \nmyp \subset \Sigma_\epsilon$, then $\spec \myq \subset \ball(0,1-\sqrt{\epsilon})$.
\end{theorem}

	\begin{proof}

	  To show that $\spec\nmyp\subset \Sigma_\epsilon \Rightarrow \spec\myq\subset \mathcal{B}\left(0,1-\sqrt{\epsilon}\right)$, we will prove the contrapositive
          \begin{align}\label{a:caoC}\spec\myq \cap \C \setminus \mathcal{B}\left(0,1-\sqrt{\epsilon}\right) \neq \emptyset \Rightarrow \spec\nmyp\cap \C \setminus \Sigma_\epsilon \neq \emptyset.\end{align}
	
	We consider an eigenvalue $\lambda\in \spec\myq \cap \C \setminus \mathcal{B}\left(0,1-\sqrt{\epsilon}\right)$ so that $\lambda=r\left(1-\sqrt{\epsilon}\right)e^{i \bar\theta}$ for some $\bar \theta \in (0,2 \pi]$ and $r>1$. The associated eigenvalue of $\nmyp$ is $$\delta_r(\bar \theta):=\frac{\lambda^2}{(1+\alpha)\lambda-\alpha}=\frac{\left(1-\sqrt{\epsilon}\right)^2 r^2 e^{2i  \bar\theta}}{\frac{2}{1+\sqrt{\epsilon}} r\left(1-\sqrt{\epsilon}\right)e^{i  \bar\theta} -\frac{1-\sqrt{\epsilon}}{1+\sqrt{\epsilon}}}=\frac{(1-\epsilon) r^2 e^{2i  \bar\theta}}{2 r e^{i  \bar\theta}-1}.$$ 
		It is easy to check from $r>1$  that  $$ \left | \delta_r(0)\right |= \frac{(1-\epsilon)r^2}{2 r-1}>1-\epsilon ,$$
		which together with~\eqref{a:sigincl} implies  that
		$$
		\delta_r(0) \notin \Sigma_{\epsilon}.
		$$
		 Suppose that $\delta_r(\bar\theta) \in \Sigma_\epsilon$. Since the curve $\Gamma_\epsilon$ is the boundary of the compact set $\Sigma_\epsilon$,   there must be a $\theta\in (0,2\pi)$ such that
		  $$
		  \delta_r(\theta) \in \Gamma_\epsilon.
		  $$
		  In other words,  there is $u,v\in \C$ such that $|u|=|v|=1$ and $(1-\epsilon)\frac{u^2}{2u-1}=(1-\epsilon)\frac{r^2 v^2}{2 r v-1}$. Then $r^2(2u-1)v^2-2ru^2v+u^2=0$. We consider $v$ as the unknown variable in this equation. The discriminant is $\Delta=4r^2u^2(u-1)^2$ and then $$v \in \left\{\frac{2ru^2\pm 2ru(u-1)}{2r^2(2u-1)}\right\} =\left\{\frac{u}{r},\frac{u}{r(2u-1)}\right\}.$$
	Since $|u|=|v|$, it is impossible that  $v=\frac{u}{r}$. If $v=\frac{u}{r(2u-1)}$, then by taking the module we have $|2u-1|=\frac{1}{r}$, which is absurd because $|2u-1|\geq |2u|-1=1>\frac{1}{r}$.  We thus conclude that 	 $\delta_r(\bar\theta)  \in \C \setminus \Sigma_\epsilon$ and~\eqref{a:caoC} is proved.
\end{proof}

\begin{remark}\rm
In Theorem~\ref{dSigmaAcc}, we will give an analysis for acceleration of  arbitrary degree $d\geq 2$, which recovers   Theorem~\ref{ThmSigma} for the case $d=2$ and  in addition shows that  if $\spec \myq \subset \ball(0,1-\sqrt{\epsilon})$, then $\spec \nmyp \subset \Sigma_\epsilon$.
\end{remark}

For any $r\geq 1$, denote by $\Gamma_\epsilon (r)$ the simple closed curve defined by the parametric equation $\theta \mapsto  \delta_r(\theta)
,\theta\in (0,2\pi]$, and denote by  $ \Sigma_\epsilon(r)$ the region enclosed by $\Gamma_\epsilon (r)$.
We have the following stronger result.

\begin{theorem}\label{ThmSigmastronger}
\textcolor{blue}{Let $\epsilon \in (0,1)$, $\nmyp$ be a $n\times n$ complex matrix, $\myq$ be defined as in~\eqref{a:myq} with}
 $\alpha=(1-\sqrt{\epsilon})/(1+\sqrt{\epsilon})$ and $r\geq 1$.
 If $\spec \nmyp \subset \Sigma_\epsilon(r)$, then $\spec \myq \subset \ball(0, r(1-\sqrt{\epsilon}))$.
\end{theorem}

\textcolor{blue}{
An ingredient of the proof of Theorem~\ref{ThmSigma} was to show that for any $r>1$, the curve $\Gamma_\epsilon (r)$ 
does not intersect with the curve $\Gamma_\epsilon$. 
In a similar way, we can prove the above~\Cref{ThmSigmastronger}, by showing that
the curve $\Gamma_\epsilon(r)$ does not intersect with the curve $\Gamma_\epsilon(r')$
for any distinct $r\geq 1$ and $r'\geq 1$.
}
\begin{remark}\rm\label{RmkSigmastronger}
Let $0<\gamma\leq 1$. An equivalent statement of Theorem~\ref{ThmSigmastronger} is as follows: if  $\spec \nmyp \subset \Sigma_\epsilon\left(\frac{1-\gamma{\sqrt{\epsilon}}}{1-\sqrt{\epsilon}}\right)$, then $\spec \myq \subset \ball(0, 1-\gamma\sqrt{\epsilon})$.  
This implies that the asymptotic rate can be of order $1-\Omega(\sqrt{\epsilon})$ if the spectrum  of $P$ is sufficiently close to $\Sigma_\epsilon$.
  We illustrate this result in Figure~\ref{fig:sigmar} with the example of $\epsilon=0.01$ and $\gamma=0.5$. 
\end{remark}

	\begin{figure}[htbp]
	\label{fig:sigmar}
	\centering
	\subfigure[]{
	\label{fig:sigmara}
\includegraphics[width=0.45\textwidth]{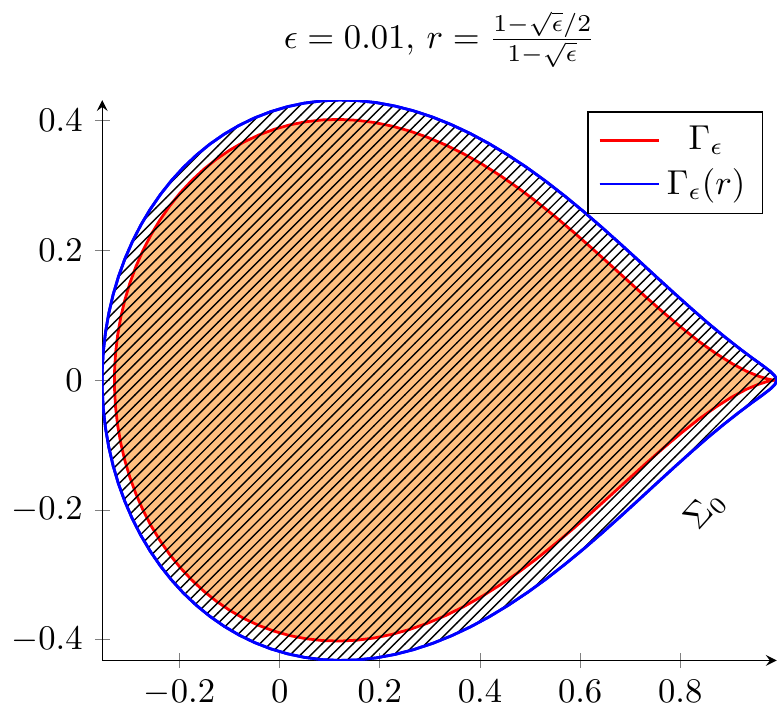}
}
\subfigure[]{
	\label{fig:sigmarb}
	\includegraphics[width=0.45\textwidth]{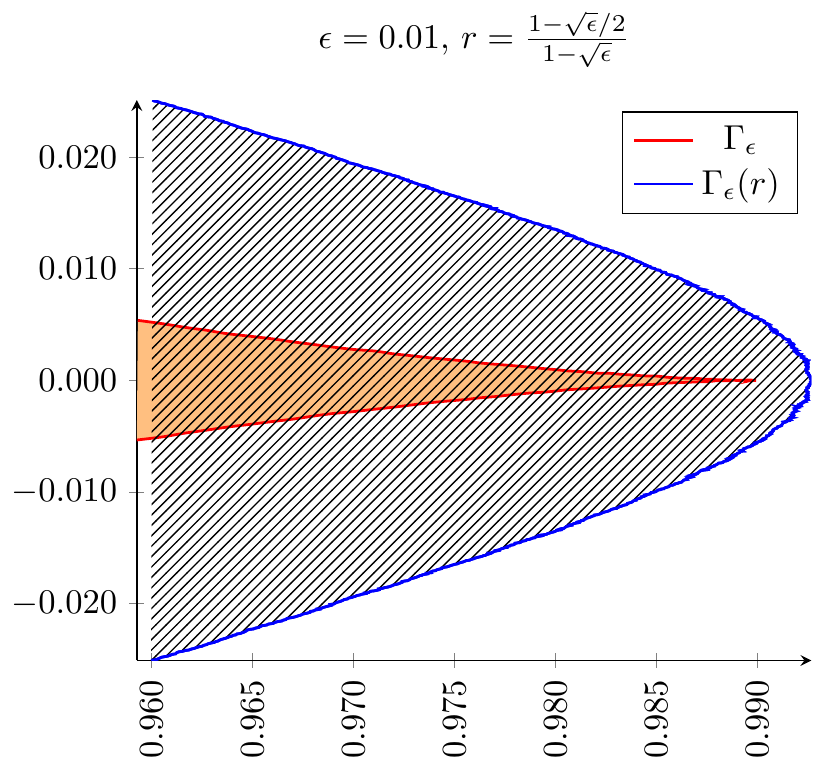}
}
\caption{Illustration of the curve $\Gamma_\epsilon$ (the curve in red) and its enclosed region $\Sigma_\epsilon$ (the region in orange), and the curve $\Gamma_\epsilon(r)$ (the curve in blue) and its enclosed region $\Sigma_\epsilon(r)$ (the dashed region). \Cref{fig:sigmara} is a zoom of~\Cref{fig:sigmarb}. }
\end{figure}

\subsection{Enlargement of the accelerable region by damping}\label{sec:thickness}
In this subsection we consider the effect of the Krasnosel'ski\u\i-Mann damping parameter $\beta \in (0,1]$
The following corollary, which is immediate from \Cref{ThmSigma}, determines the accelerable region for the spectrum of the initial matrix $\origp$.
\begin{corollary}\label{cor:bta}
	If there is $\beta\in (0,1]$ such that $\spec P_\beta \subset \Sigma_{\epsilon}$, then AVI algorithm~\eqref{AVIalphabeta} with the parameters $\beta$ and $\alpha=\frac{1-\sqrt{\epsilon}}{1+\sqrt{\epsilon}}$ converges  with an  asymptotic rate no greater than  $1-\sqrt{\epsilon}$, i.e., ~\eqref{a:sqrt} holds.
\end{corollary}


Based on \Cref{cor:bta}, we now look for a radius $r>0$ such that if $\spec P \subset \ball(0,r) \cup [-1+\epsilon,1-\epsilon]$, then there is a scaling parameter $\beta \in (0,1]$ such that $\spec P_\beta \subset  \Sigma_{\epsilon'} $, for some $\epsilon'>0$
, with the goal of achieving
 an accelerated asymptotic rate $1-\Omega(\sqrt{\epsilon})$
.

We start by giving a disk and a part of the real line which are contained  in $\Sigma_0$.
\begin{lemma}\label{l:1o3}
We have
$$
\ball\left(\frac{1}{3},\frac{1}{3}\right) \cup \left[-\frac{1}{3},1\right]\subset \Sigma_0.
$$
\end{lemma}
\begin{proof}
The boundary of $\Sigma_0$ intersects with the real axis  at $(1,0)$ and $(-1/3,0)$. Thus $\left[-\frac{1}{3},1\right]\subset \Sigma_0$.
For any $\theta\in (0, 2\pi]$, we have
\begin{align*}
&\left |\frac{e^{2i \theta	}}{2e^{i \theta}-1}-\frac{1}{3}\right |^2-\frac{1}{9}= \left |\frac{3 e^{2i \theta}-2e^{i \theta}+1}{3(2e^{i \theta}-1)}\right|^2-\frac{1}{9}
\\ & = \frac{ |3 e^{2i \theta}-2e^{i \theta}+1|^2-|2e^{i \theta}-1|^2 }{9 |2e^{i \theta}-1|^2 }\\ &=\frac{ 14-16\cos(\theta)+6 \cos(2 \theta)-5+4\cos(\theta)}{9 |2e^{i \theta}-1|^2} \\& = \frac{4\cos^2(\theta)-4\cos(\theta)+1}{9 |2e^{i \theta}-1|^2}   =\frac{(2\cos(\theta)-1)^2} {9 |2e^{i \theta}-1|^2}\geq 0 .
\end{align*}
Thus the boundary of $\Sigma_0$ does not intersect the interior of the disk $\ball (\frac{1}{3},\frac{1}{3})$. Since $0\in \Sigma_0 \cap  \ball (\frac{1}{3},\frac{1}{3})$,  the disk 
$\ball (\frac{1}{3},\frac{1}{3})$ is entirely contained in   $\Sigma_0$.
\end{proof}

The following result shows that if the spectrum of the initial matrix $P$
belongs to a ``flying saucer'' shaped region of the complex plane (see \Cref{fig:fsa} for illustration),
the AVI algorithm does converge with an asymptotic rate $1-\Omega(\sqrt{\epsilon})$. 
\begin{theorem}\label{th:flyingsaucer}
  If $\spec P\subset \mathcal{B}(0,\frac{1-\epsilon}{2})\cup [-1+\epsilon,1-\epsilon]$, then by setting $\beta=\frac{2}{3-\epsilon}$ and \begin{equation}\label{eq:alpha}\alpha=\frac{1-\sqrt{2\epsilon/(3-\epsilon)}}{1+\sqrt{2\epsilon/(3-\epsilon)}},\end{equation} the iterates of algorithm AVI~\eqref{AVIalphabeta}  satisfy $$  \limsup_{k\to\infty}\|x_k-x_*\|^{1/k} \leq 1-\sqrt{\frac{2\epsilon}{3}}.$$
\end{theorem}
\ifX
\begin{proof}For any $\beta \in (0,1]$, the spectrum of $P_{\beta}$ is the image of the spectrum of $P$  by the homothety $H^{\beta}:=z\mapsto 1-\beta + \beta z$ of center $1$ and ratio ${\beta}$. 
Note that $\beta=\frac{2}{3-\epsilon}$ satisfies
$$
1-\beta= \frac{\beta(1-\epsilon)}{2}=\frac{1-\beta \epsilon}{3}.
$$
Hence the image of $ \mathcal{B}(0,\frac{1-\epsilon}{2})\cup [-1+\epsilon,1-\epsilon]$ by  the homothety $H^\beta$  is
$$
\mathcal{B}\left (\frac{1-\beta \epsilon}{3},\frac{1-\beta \epsilon}{3}\right )\cup \left [-\frac{1-\beta\epsilon}{3},1-\beta\epsilon\right].
$$
See~\Cref{fig:fsb} for an illustration.
In view of~\Cref{l:1o3}, this region is contained in $\Sigma_{\beta\epsilon}$. It follows that $\spec P_{\beta}\subset \Sigma_{\beta\epsilon}$ and the statement follows by applying \Cref{cor:bta}.
\end{proof}
\fi

	\begin{figure}[htbp]
	\label{fig:sigmar-bis}
	\centering
	\subfigure[]{
	\label{fig:fsa}
	\includegraphics[width=0.45\textwidth]{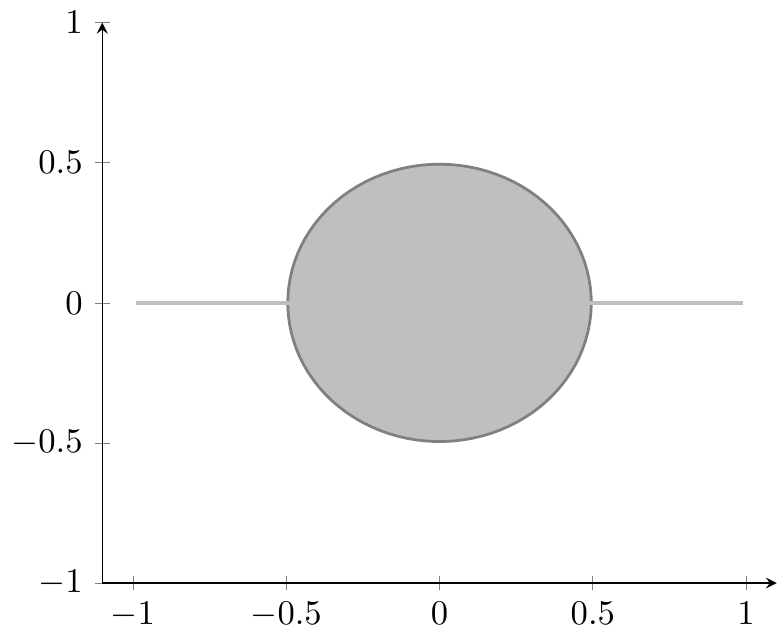}
}
\subfigure[]{
	\label{fig:fsb}
	\includegraphics[width=0.45\textwidth]{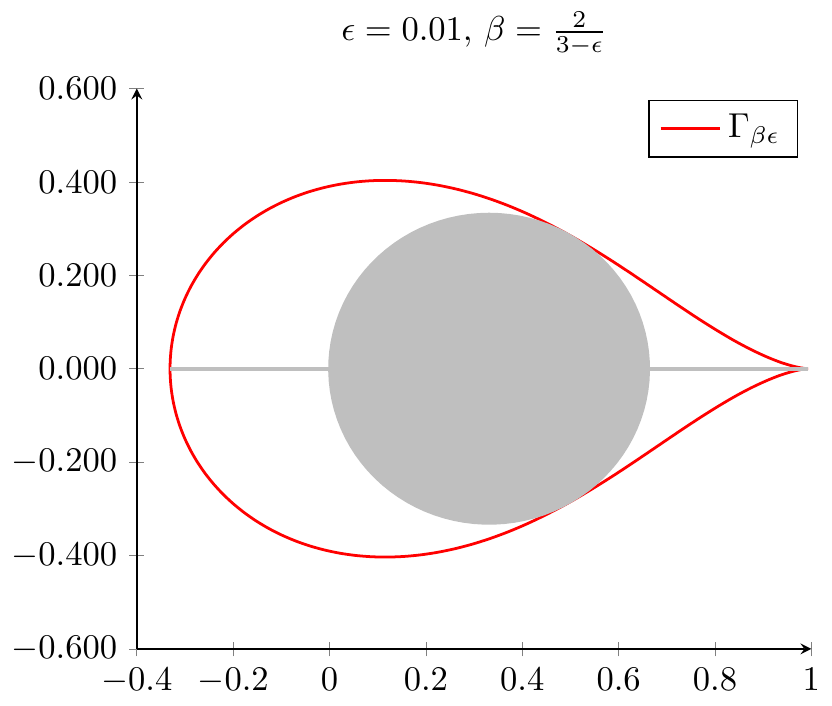}
}
\caption{(a):
  the region $\mathcal{B}\left(0,\frac{1-\epsilon}{2}\right)\cup [-1+\epsilon,1-\epsilon]$ (the flying saucer shaped region in grey).
  (b):
  the curve $\Sigma_{\beta\epsilon}$ (the curve in red) and the image of $\mathcal{B}\left(0,\frac{1-\epsilon}{2}\right)\cup [-1+\epsilon,1-\epsilon]$ by the homothety $H^\beta$. }
\end{figure}
\begin{remark}\rm For $0<\epsilon<\frac{1}{3}$,
the flying saucer shaped region $\mathcal{B}\left(0,\frac{1-\epsilon}{2}\right)\cup [-1+\epsilon,1-\epsilon]$
can not be included in $\Sigma_{0}$ and \Cref{cor:bta} is not applicable. However, the homothety  $H^\beta$ with $\beta=\frac{2}{3-\epsilon}$  sends this region inside  $\Sigma_{\beta\epsilon}$, whence an accelerated asymptotic rate.
\end{remark}
\begin{remark}\rm
In the special  case when  $\spec P\subset [-1+\epsilon,1-\epsilon]$, a similar result  was established in~\cite{julien}.  Translated with our notations, Theorem 5.1 in~\cite{julien} proved an asymptotic rate $1-\sqrt{{\epsilon}/(2-\epsilon)}$ by setting $\alpha={1}/{(2-\epsilon)}$ and $\beta=\frac{1-\sqrt{\epsilon/(2-\epsilon)}}{1+\sqrt{\epsilon/(2-\epsilon)}}$ in~\eqref{AVIalphabeta}.  
\end{remark}
We complement~\Cref{th:flyingsaucer} by  showing the optimality of the radius $\frac{1-\epsilon}{2}$ in the sense described by the following lemma. 
\begin{lemma}\label{lemCircle}
	The largest radius $r\geq 0$, for which there exists $\beta\in (0,1]$ such that 
	 $H^{\beta}(\mathcal{B}(0,r))\subset \Sigma_0$, is $r=\frac{1}{2}$ and it corresponds to the choice $\beta=\frac{2}{3}$.
\end{lemma}
\ifX
\begin{proof}
Applying the homothety $H^{\beta}$ to $\mathcal{B}(0,r)$ leads to the ball $\mathcal{B}(1-\beta,\beta r)$. We thus look for the largest $r$ such that  $\mathcal{B}(1-\beta,\beta r)\subset \Sigma_0$ for some $\beta\in(0,1)$.
	We notice that $r>1$ is not possible, because  for any $r>1$ we have $1+\beta (r-1)>1$, which is outside $\Sigma_0$.
	 
	Now, we suppose that $0 \leq r \leq 1$ and there is $\beta\in(0,1)$ such that $\mathcal{B}(1-\beta,\beta r) \subset \Sigma_0$.  We consider the line $(D_r)$ of the complex plane passing through the point of coordinates $(1,0)$ and tangent to the upper half of the circle $\mathcal{B}(0,r)$. This line is given by the equation $$y=\frac{r}{\sqrt{1-r^2}}(1-x).$$  
	Note that $(D_r)$ is invariant by the homothety $H^{\beta}$ and thus is also tangent to $\mathcal{B}(1-\beta,\beta r)$.  Thus $(D_r) $ must intersect with  $ \Sigma_0$ at a  point  other than $(1,0)$, see~\Cref{fig:lineD}
	for an illustration.
	The curve $\Gamma_0$ is given by $$\theta \mapsto \frac{e^{2i\theta}}{2 e^{i\theta}-1}
	=\frac{2 \cos(\theta)-\cos(2 \theta)}{5-4\cos(\theta)}+i \frac{2 \sin(\theta)-\sin(2 \theta)}{5-4\cos(\theta)}, \enspace\theta\in(0,2\pi].$$ 
	Let  $\theta\in (0,2\pi)$ such that $ (x_0,y_0):=\left(\frac{2 \cos(\theta)-\cos(2 \theta)}{5-4\cos(\theta)}, \frac{2 \sin(\theta)-\sin(2 \theta)}{5-4\cos(\theta)}\right)$ lies in $(D_r)\cap \Sigma_0$. Then 
	$$
	\frac{r}{\sqrt{1-r^2}}=\frac{y_0}{1-x_0}=\frac{2\sin(\theta)(1-\cos(\theta))}{5-6\cos(\theta)+2\cos^2(\theta)-1} =\frac{\sin(\theta)}{2-\cos(\theta)}\enspace.
	$$
	We can easily prove that:
		\[
		\max_{\theta\in (0,2\pi)} \frac{\sin(\theta)}{2-\cos(\theta)}= \frac{\sqrt{3}}{3} \enspace.
		\]
		Hence, 
	$$
	\frac{r}{\sqrt{1-r^2}} \leq \frac{\sqrt{3}}{3},
	$$
	which implies that $r\leq 1/2$.
	
	When $r=1/2$, we let $\beta=2/3$. Then  the image of $\ball(0,r)$ by the homothety $H^\beta$  is $\ball (1/3,1/3)$, which by \Cref{l:1o3} is contained in $\Sigma_0$.
\end{proof}

\begin{figure}[htbp]
	\label{fig:lineD}
	\centering
	\subfigure[]{
	\label{fig:linea}
\includegraphics[width=0.45\textwidth]{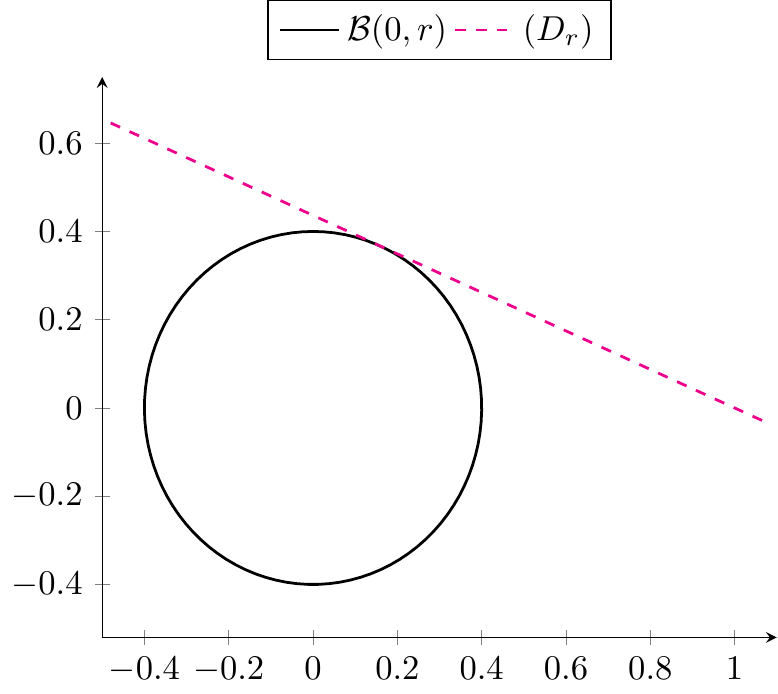}
}
\subfigure[]{
	\label{fig:lineb}
\includegraphics[width=0.45\textwidth]{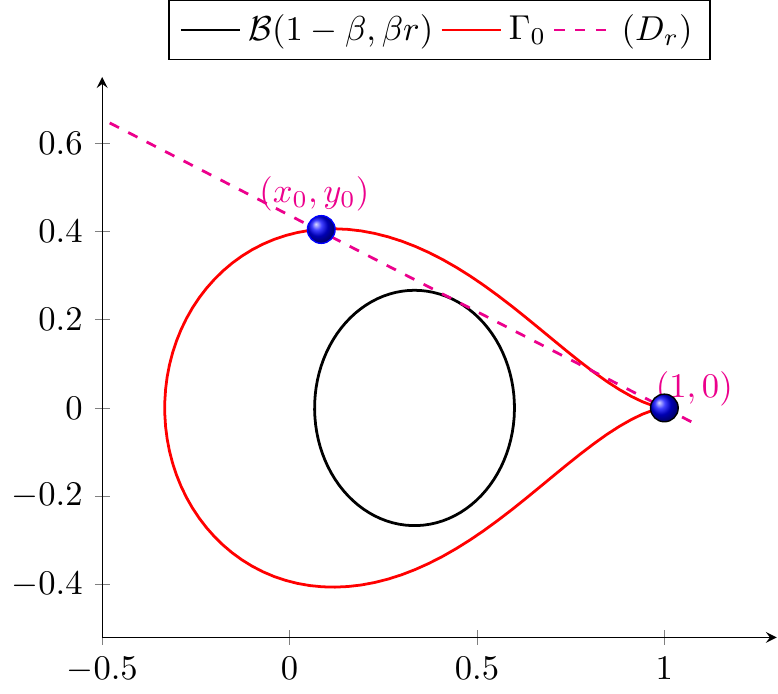}
}
\caption{(a): 
  the line $(D_r)$ is tangent to the boundary of $\mathcal{B}\left(0,r\right)$. 
  (b): the line $(D_r)$ is tangent to the boundary of $\mathcal{B}\left(1-\beta,\beta r\right)$, which is contained in $\Sigma_0$. The line  $(D_r)$ intersects $\Sigma_0$ at $(1,0)$ and $(x_0,y_0)$.}
\end{figure}

\section{Analysis of Accelerated Value Iteration of degree $d$ 
}\label{sec:d3}

In this section we consider the acceleration scheme $d$A-VI~\eqref{dAVI} of any order $d\geq 2$. 
Hereinafter, $\alpha=(\alpha_0, \cdots, \alpha_{d-2})\in \R^{d-1}$ denotes the vector of parameters required in~\eqref{dAVIy}.   
We shall extend the previous results for $d=2$ to arbitrary $d\geq 2$. That is,  with an appropriate choice of  $\alpha$, and under an assumption on the shape of the complex spectrum of $\origp$, the asymptotic rate of~\eqref{dAVI} can  be   $1-\epsilon^{1/d}$. We refer to Remark~\ref{rem:AA} for a discussion on the connection between the $d$A-VI~\eqref{dAVI} and Anderson acceleration.



\subsection{Parameters}\label{sec:pAR}
We show how to select the parameters $\alpha=(\alpha_0, \cdots, \alpha_{d-2})$
in~\eqref{dAVIy} to obtain an acceleration of any order $d\geq 2$. For the sake of simplicity we let $\beta=1$.
Then $z_k=y_k-x_*$ satisfies the following system of linear equations:
\[
\begin{pmatrix}
  z_{k+1}\\
z_k\\
\vdots\\
z_{k-d+2}
\end{pmatrix}=\myqd \begin{pmatrix}
z_{k}\\
z_{k-1}\\
\vdots\\
z_{k-d+1}
\end{pmatrix}
\]
where 
\[
\myqd:=\begin{pmatrix} 
(1+\alpha_{d-2}+\cdots+\alpha_0){\nmyp} & -\alpha_{d-2} {\nmyp} & \cdots & -\alpha_0 {\nmyp}\\
I & 0 & \cdots & 0\\
\vdots & \ddots & & \vdots\\
0 & \cdots & I & 0
\end{pmatrix},
\]
We introduce the following rational function of degree $d$ defined by 
\begin{align}\label{a:phiad}\phi_{\alpha,d}(\lambda)= \frac{\lambda^d}{U(\lambda) },\end{align} where $U(\cdot):\C\rightarrow \C$ is the  polynomial of degree $d-1$ given by:
\begin{equation*}\label{alpha_d}
U(\lambda)=(1+\alpha_{d-2}+\cdots+\alpha_0)\lambda^{d-1}-\alpha_{d-2}\lambda^{d-2}-\cdots-\alpha_0 \enspace .
\end{equation*}
The polynomial $U$ satisfies $U(1)=1$. The following
standard result, which is proved as~\Cref{SpecQ} above,
relates the eigenvalues of $\myqd$ with those of $\nmyp$. 
\begin{lemma}\label{lem-d-SpecQ}
	$\lambda$ is an eigenvalue of $\myq$ if and only if there exists an eigenvalue $\delta$ of $\nmyp$ such that $\delta=\phi_{\alpha,d}(\lambda)$.  In other words, $$\spec \myqd = \phi_{\alpha, d}^{-1}(\spec \nmyp).$$
\end{lemma}


\ifX{

\textcolor{blue}{We want to choose the vector of parameters $\alpha$ that leads to the smallest possible spectral radius for $\myqd$, in order to obtain the smallest asymptotic rate for~\eqref{dAVI}, like in the case of AVI (i.e.\ $d=2$).
}

\begin{lemma}\label{l:bestc}
The best choice
of the parameters $\alpha_0, \cdots, \alpha_{d-2}$ that minimizes the maximum of the moduli of the preimages of $1-\epsilon$ by $\phi_{\alpha, d}$ is:
\begin{equation}\label{alpha_i}
\alpha_i=\binom{d}{i}\frac{(\epsilon^{1/d}-1)^{d-i}}{ (1-\epsilon)}, \quad \forall \quad i=0,\cdots, d-2 \enspace ,
\end{equation}
and it corresponds to the following rational function
\begin{align}\label{a:phistard}
\phi^*_{d}(\lambda)=\frac{(1-\epsilon)\lambda^d}{\lambda^d-(\lambda-(1-\epsilon^{1/d}))^d} \enspace .
\end{align}
\end{lemma}
\begin{proof} 
It is easy to verify that with the choice of $\alpha_0,\ldots, \alpha_{d-2}$ in~\eqref{alpha_i},  
\[
U(\lambda)=\frac{1}{1-\epsilon}\left(\lambda^d-\left(\lambda-(1-\epsilon^{1/d})\right)^d\right) \enspace ,
\]
and thus it leads to the rational function~\eqref{a:phistard}. In addition, 
$
\phi_d^*(\lambda)=1-\epsilon
$
if and only if $\left(\lambda-(1-\epsilon^{1/d})\right)^d=0$, from which we deduce that the  maximal moduli of the preimages of $1-\epsilon$ by $\phi_d^*$ is $1-\epsilon^{1/d}$.

Let $\lambda_1,\lambda_2,\cdots,\lambda_d$ be the solutions of $\phi_{\alpha, d}(\lambda)=1-\epsilon$ satisfying $\max_i |\lambda_i|\leq 1-\epsilon^{1/d}$.
Then $\lambda^d-(1-\epsilon)U(\lambda)=\prod_{i}(\lambda-\lambda_i)$ for all $\lambda\in \C$.  
By taking $\lambda=1$ we obtain that $\prod_{i}(1-\lambda_i)=\epsilon$. We have 
\begin{align*}\label{Ineqlambda}
\epsilon\leq (1-\max_i|\lambda_i|)^d \leq \prod_{i}(1-|\lambda_i|)\leq \prod_{i}|1-\lambda_i|=\epsilon \enspace .
\end{align*}
Therefore for all $i$, $\lambda_i = 1-\epsilon^{1/d}$ and $\phi_{\alpha, d}$ is exactly $\phi_d^*$.
\end{proof}
}
\fi




In the following, we consider the scheme~\eqref{dAVI} implemented with the special choice of the parameters $\{\alpha_i: i=0,\cdots,d-2\}$ given in \cref{alpha_i}. 
We want to generalize the characterization of the accelerable region $\Sigma_\epsilon$ for the AVI algorithm
to get the largest accelerable  region  for the $d$A-VI algorithm. 
For this purpose, for any $\epsilon \geq 0$ and $d\geq 2$, let $\Gamma_{\epsilon,d}$ be the simple closed curve defined by the parametric equation:
	\begin{equation}\label{eq:Gammaepsilond}
	\theta \mapsto \frac{ (1-\epsilon)\e^{i d \theta} }{\e^{i d \theta}-(\e^{i \theta}-1)^d} , \quad \theta \in \left(\pi - \frac{2\pi}{d},\pi + \frac{2\pi}{d}\right]\enspace.
	\end{equation}
	See an illustration in~\Cref{fig:Gamma4} for $\epsilon=0$ and $d=4$.
	\begin{figure}[htbp]
\label{fig:Gamma4}
	\centering
	\subfigure[curve $\Gamma_{\epsilon,d}$ ]{
	\label{Gammad4}
		\includegraphics[width=0.45\textwidth] {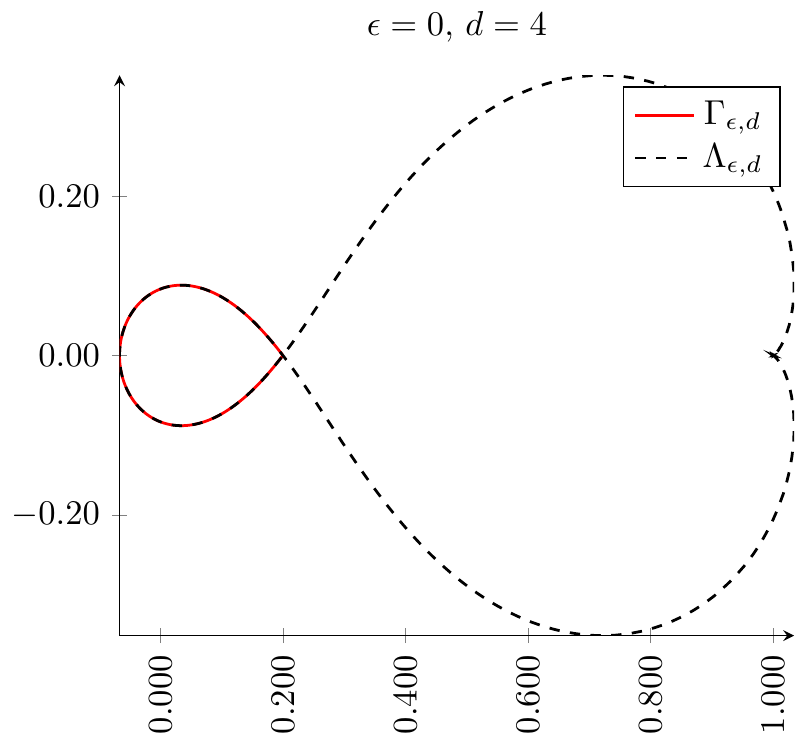}	
	}\subfigure[region $\Sigma_{\epsilon,d}$]{
	\label{Sigmad4}
		\includegraphics[width=0.45\textwidth] {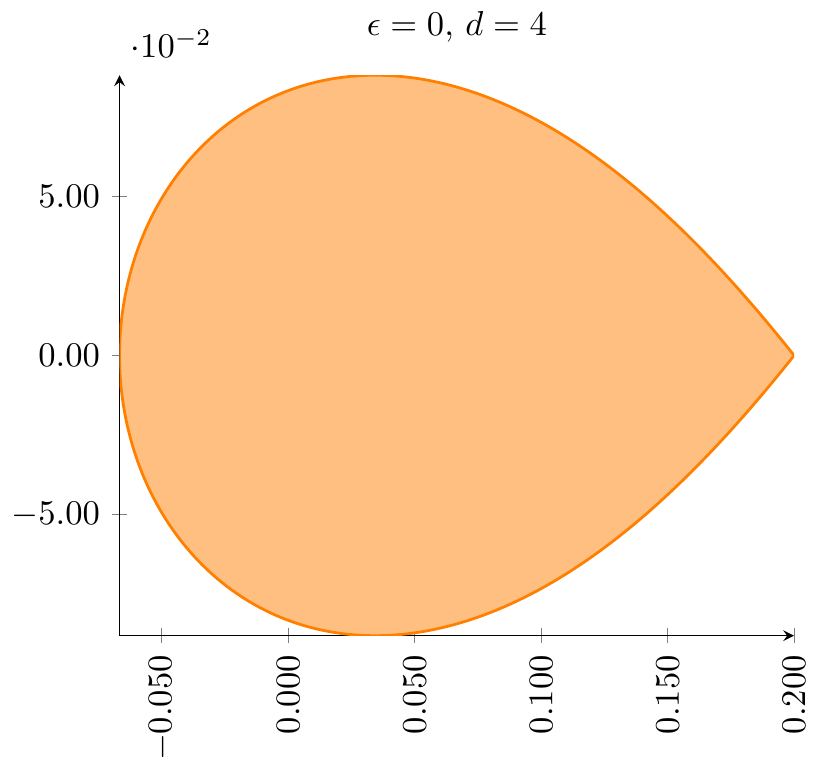}	
	}
	\caption{Ilustration of  the curve $\Gamma_{\epsilon,d}$ (the curve in red in \Cref{Gammad4}) and its enclosed region  $\Sigma_{\epsilon,d}$ (\Cref{Sigmad4}). The dashed curve in \Cref{Gammad4}
	corresponds to $\left\{ \frac{ (1-\epsilon)\e^{i d \theta} }{\e^{i d \theta}-(\e^{i \theta}-1)^d}: \theta \in (0,2\pi])\right\}$. }
	\end{figure}
	Denote by $\Sigma_{\epsilon,d}$ the compact set delimited by the  simple closed curve $\Gamma_{\epsilon,d}$. 
The following theorem identifies
conditions on the spectrum of the initial matrix $P$ which guarantee that
the $d$A-VI algorithm converges asymptotically with a rate
$1-\epsilon^{1/d}$.

\begin{theorem}\label{dSigmaAcc}
  Choosing the parameters $(\alpha_0,\dots, \alpha_{d-1})$ as in \cref{alpha_i}, we get
  that
$$\spec\myqd \subset \ball(0,1-\epsilon^{1/d}),$$ if and only if $\spec\nmyp\subset \Sigma_{\epsilon,d}\cup \{1-\epsilon\}$.
\end{theorem} The proof is given in the next subsection.

\subsection{Proof of Theorem~\ref{dSigmaAcc}}\label{subsec:proof}
\begin{lemma}\label{l:pdl1}
$\spec \myqd \subset \ball(0,1-{\epsilon^{1/d}})$ if and only if
\begin{align}\label{a:specP2}
\spec P \subset \{(1-\epsilon) z: \psi_d^{-1}(z)\subset \ball (0,1) \}\enspace,
\end{align}
where $\psi_d$ is the rational function defined by
\begin{equation*}
\psi_d (\lambda) = \frac{\lambda^d}{\lambda^d-(\lambda-1)^d} \enspace .
\end{equation*}
\end{lemma}
\begin{proof}
We note from~\Cref{lem-d-SpecQ} that $\spec \myqd \subset \ball(0,1-{\epsilon^{1/d}})$ if and only if
\begin{align}\label{a:specP}
\spec P \subset \{ z: (\phi^*_d)^{-1}(z)\subset \ball (0,1-{\epsilon^{1/d}}) \}\enspace.
\end{align}
We note the following property:
\begin{equation}\label{PropertyPhi}
  \phi^*_{d}((1-\epsilon^{1/d})\lambda)=\frac{(1-\epsilon)\lambda^d}{\lambda^d-(\lambda-1)^d} = (1-\epsilon)\psi_d (\lambda) 
  \enspace ,
\end{equation}

Hence~\eqref{a:specP} is equivalent to~\eqref{a:specP2}.
\end{proof}
We next give a description of the following set.
\begin{align}\label{a:stabled}
\cS:=\left\{z\in \C: \psi_d^{-1} (z) \subset \ball(0,1)\right\}.
\end{align}
We shall need to define 
$$
\cQ:=\bigcap_{k=0}^{d-1} e^{\frac{2k\pi i}{d}} H,
$$
where
$$
H:=\{w\in \C: \real(w)\leq 1/2\},$$
is the  half-plane containing all the complex numbers with real part smaller than $1/2$, and  $e^{\alpha i} H$  denotes the
halfspace obtained by rotating $H$ of angle $\alpha$.
\begin{lemma}\label{l:erdd}
\begin{align}
\cS= \left\{\frac{1}{1-\frac{1}{z^d}}: z\in \cQ\right\}\cup\{1\}.
\end{align}
\end{lemma}
\begin{proof}
We define two self-maps of the extended complex plane $\bar \C$:
\begin{align}
&f_1(\lambda):=
\frac{\lambda}{\lambda-1} \enspace ,\\
& f_2(\lambda):=\lambda^d \enspace .
\end{align}
Note that
\[
f_1(\lambda) = 1 + \frac{1}{\lambda-1},
\]
which entails that $f_1$ is an inversion of center $1$.
In particular, $f_1\circ f_1(\lambda)=\lambda$ for any $\lambda \in \bar \C$.
It is easy to see that
\begin{align}\label{a:psif12}
\psi_d(\lambda)=f_1\circ f_2 \circ f_1 (\lambda),\enspace \forall
\lambda\in \bar\C.
\end{align}
Hence  we know that
\begin{equation}\label{q:cS}
\begin{array}{ll}
\cS&=\left\{z\in  \C: \psi_d^{-1}(z)\subset \ball(0,1)\right\}\\
&=\left\{z\in \C: f^{-1}_1 \left(f_2^{-1} \left(f_1^{-1}(z)\right)\right) \subset \ball(0,1)\right\}\\ 
&=\left\{z\in \C: f_2^{-1} \left(f_1^{-1}(z)\right)\subset  f_1(\ball(0,1))\right\}\\
&=\left\{f_1(w)\in \C: f_2^{-1} \left(w\right)\subset  f_1(\ball(0,1))\right\},
\end{array}
\end{equation}
where the second equality used~\eqref{a:psif12}, the third equality relies on the bijection property of $f_1$ and the last equality applies the change of variable $w=f_1^{-1}(z)$.

 Now we characterize the set $f_1(\ball(0,1))$. Note that $z= f_1(w)$ if and only if $\frac{1}{z}+\frac{1}{w}=1$. Thus there is $w\in \ball(0,1)$ such that $z=f_1(w)$   if and only if  
	$|1-\frac{1}{z}|\geq 1$. We then deduce that
	\begin{align}\label{a:f1B01}
	f_1(\ball(0,1))=\left\{w \in  \bar\C: |w-1|\geq |w|\right\}.
	\end{align}
	Note that
	$$
	\left\{w \in  \C: |w-1|\geq |w|\right\}=\{w\in \C: \real(w)\leq 1/2\}=H.
	$$Indeed, it is known that a circle passing through the center of an inversion
is sent to a line by this inversion, and the disk
delimited by the circle is send to a half-plane. We conclude that
 \begin{align}\label{a:f1B012}
 f_1(\ball(0,1))=H\cup \{\infty\}.
 \end{align}
Plugging~\eqref{a:f1B012} into~\eqref{q:cS} we obtain
\begin{equation}\label{q:cS2}
\begin{array}{ll}
\cS=\left\{f_1(w)\in \C: f_2^{-1} \left(w\right)\subset  H\cup \{\infty\}\right\}.
\end{array}
\end{equation}
It remains to characterize 
the set
$$
\left\{w\in \bar\C: f_2^{-1}(w) \subset H\cup \{\infty\}\right \}=\left\{w\in \C: f_2^{-1}(w) \subset H\right \}\cup\{\infty\}.
$$
Define:
	\begin{align}\label{a:defQ-old}
	\bar \cQ:=\left\{z\in \C: f_2^{-1}(f_2(z)) \subset H\right \}.
	\end{align}
It is easy to see that:
$$
\left\{w\in \C: f_2^{-1}(w) \subset H\right \}=\left\{f_2(z): z\in \bar\cQ\right \}
$$
 It follows that
\begin{align}\label{a:HQ}
\left\{w\in \bar\C: f_2^{-1}(w) \subset H\cup \{\infty\}\right \}=\left\{ f_2(z): z\in  \bar \cQ\right\}\cup \{\infty\}.
\end{align}
Finally plugging~\eqref{a:HQ} into~\eqref{q:cS2}  we obtain that
\begin{align}
\cS=\left\{f_1(f_2(z))\in \C: z\in \bar\cQ\right\}\cup \{1\}= \left\{\frac{1}{1-\frac{1}{z^d}}: z\in \bar\cQ\right\}\cup\{1\}.
\end{align}
Since for any $z\in\C$, 
$$
	f_2^{-1}\left(f_2 (z)\right)=\left\{z, e^{-\frac{2\pi i}{d}}z,\dots, e^{-\frac{2(d-1)\pi i}{d}}z\right\},
	$$
	we obtain
	\begin{align}\label{a:defQ}
	\bar\cQ=\left\{ z \in\C: \left\{z, e^{-\frac{2\pi i}{d}}z,\dots, e^{-\frac{2(d-1)\pi i}{d}}z\right\}\subset H\right \}.
	\end{align}
	Therefore, $\bar\cQ$ is actually the intersection of $d$  halfspaces obtained by rotating $H$ of angles $\frac{2k\pi}{d}$ for $k=0,\dots ,d-1$. Namely,  
$$
\bar\cQ=\cQ.
$$

~
\end{proof}

\begin{remark}\rm\label{rm:Q}
	For $d=2$, $\cQ$ is the set of complex numbers with real part in $[-1/2,1/2]$. For $d\geq 3$, $\cQ$ is a regular polygon with $d$ vertices which circumscribes the disk $\ball(0,1/2)$, see \Cref{fig:Q} for illustration from $d=2$ to $d=5$.  In particular we have $\ball(0,1/2)\subset \cQ\subset \ball(0,1/(2\cos(\pi/d)))$ and $\cQ$ asymptotically approximates $\ball(0,1/2)$ when $d\rightarrow +\infty$. It follows that
	$$
	\left\{\frac{1}{1-\frac{1}{z^d}}: z\in \ball(0,1/2)\right\} \cup \{1\} \subset \cS\subset \left\{\frac{1}{1-\frac{1}{z^d}}: z\in \ball(0,1/(2\cos(\pi/d)))\right\}\cup \{1\} \enspace.
	$$
	 Note that  for any $a>1$,
	$$
        \left\{\frac{1}{1-\frac{1}{z^d}}: z\in \ball(0,a)\right\} =   \left\{\frac{1}{1-\frac{1}{w}}: w\in \ball(0,a^d)\right\}=  \left\{ \frac{1}{z}: |z-1|\geq a^d \right\}.
	$$
	and thus, 
	$$
	\ball(0, 1/(a^d+1))\subset  \left\{\frac{1}{1-\frac{1}{z^d}}: z\in \ball(0,a)\right\} \subset \ball(0, 1/(a^d-1)).
	$$
	This allows to deduce the following estimation of the region $\cS$.
	\begin{equation}\label{ball-Sigma}
	\ball\left(0,\frac{1}{2^d+1}\right)\cup \{1\} \subset \cS \subset \ball\left(0,\frac{1}{(2\cos(\pi/d))^d-1}\right) \cup \{1\}.
	\end{equation}
\end{remark}
\begin{figure}[htbp]
\label{fig:Q}
	\centering
	\subfigure[d=2]{
		\includegraphics[width=0.45\textwidth] {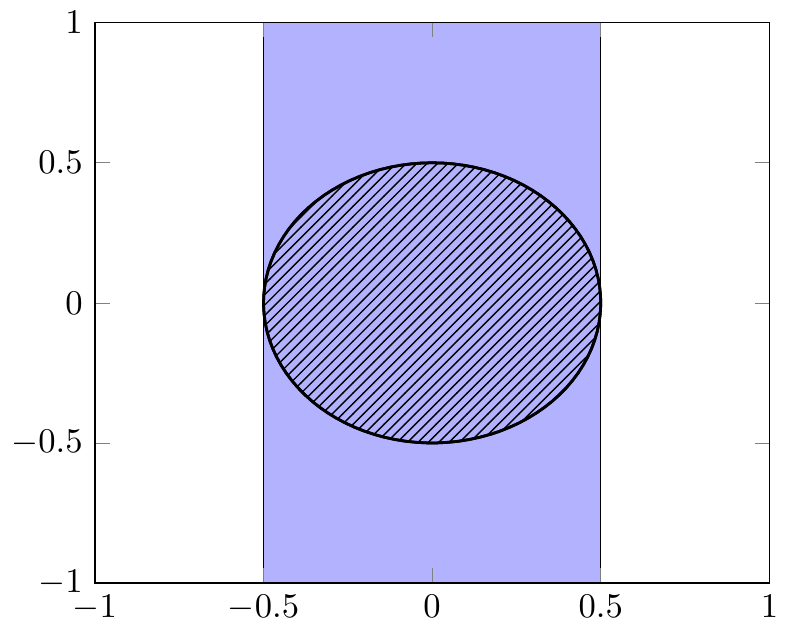}	
	}\subfigure[d=3]{
		\includegraphics[width=0.45\textwidth] {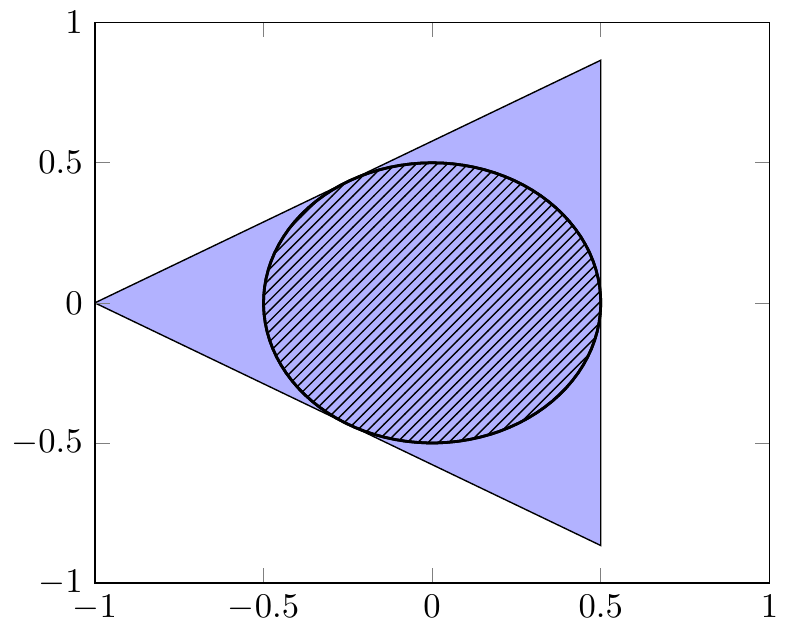}	
	}
	\subfigure[d=4]{
		\includegraphics[width=0.45\textwidth] {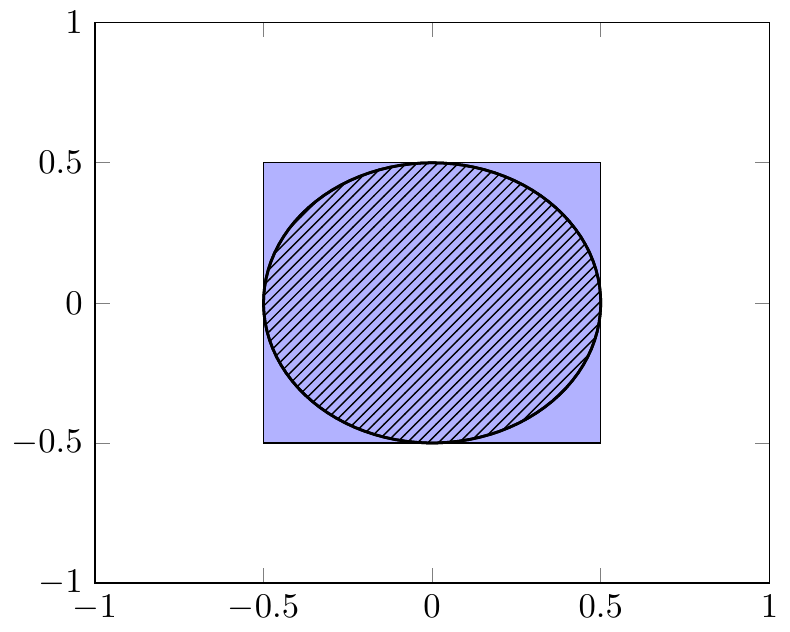}
	}
	\subfigure[d=5]{
		\includegraphics[width=0.45\textwidth] {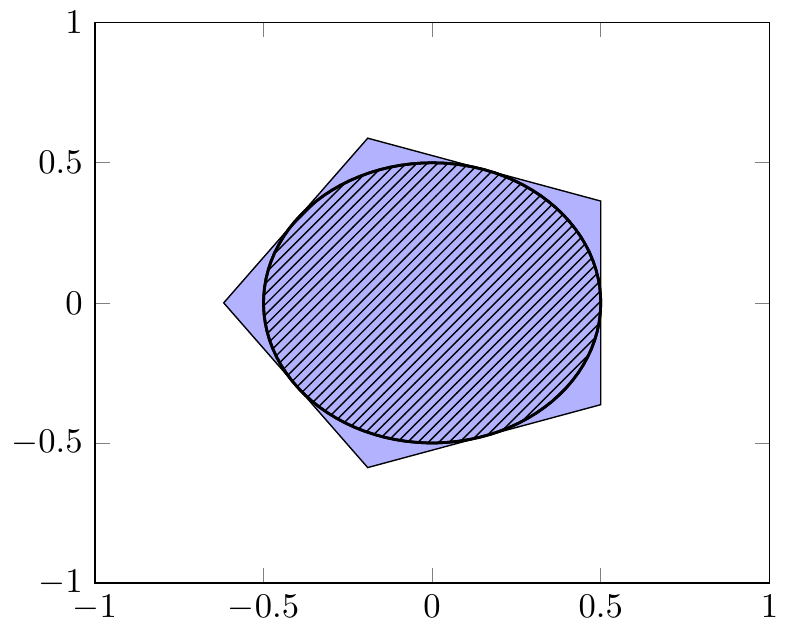}
	}
	\caption{Ilustration of  $\cQ$ (the region in blue) and the circumscribed disk $\ball(0,1/2)$ (the dashed region).}
	\end{figure}
	 Next we characterize the boundary of the accelerable region $\cS$.  \textcolor{blue}{We denote by $\Bd S$ the boundary of a set $S$}.
	\begin{proposition}\label{th-multi}We have
	\begin{equation}\label{eq-P-Sigma}
\cS=\Sigma_{0,d} \cup \{1\}
\end{equation} 
where $\Sigma_{0,d}$ is the compact set of the complex plane  delimited by the simple closed curve $\Gamma_{0,d}$ as defined in~\eqref{eq:Gammaepsilond}.
	\end{proposition}
\begin{proof}
Since $f_2$ is holomorphic and thus open, it sends the interior of $\cQ$ into the interior of $f_2(\cQ)$. It follows that $\Bd f_2(\cQ)\subset f_2(\Bd \cQ)$. 
By the continuity of $f_2$, for any $z\in \Bd\cQ$ and any $\epsilon>0$, there is $\delta>0$ such that
$$
f_2(\ball(z, \delta))\subset \ball(f_2(z),\epsilon).
$$
Since $z\in \Bd\cQ$,  $\ball(z, \delta) \cap \cQ^c\neq \emptyset$ and thus $f_2(\ball(z, \delta) ) \cap  f_2( \cQ^c)\neq \emptyset$.we note from the definition~\eqref{a:defQ} that 
$$
f_2(\cQ)\cap f_2(\cQ^c)=\emptyset.
$$
Thereby $f_2(\ball(z, \delta) ) \cap (f_2(\cQ))^c\neq \emptyset$ and $\ball(f_2(z),\epsilon)\cap (f_2(\cQ))^c\neq \emptyset$. This shows that $f_2(z)\in \Bd(f_2(\cQ))$ and thus $  f_2(\Bd \cQ)\subset\Bd f_2(\cQ)$. We thus proved that\begin{align}\label{a:Bdf2}
\Bd f_2(\cQ)=f_2(\Bd \cQ).
\end{align}
Since $f_1:\bar \C\rightarrow \bar \C$ is a homeomorphism, we know that
\begin{align}\label{a:Bdf1f2}
\Bd f_1(f_2(\cQ))=f_1(\Bd f_2(\cQ))\overset{\eqref{a:Bdf2}}{=}f_1(f_2 (\Bd \cQ)).
\end{align}
As mentioned in Remark~\ref{rm:Q}, for $d=2$,  $\cQ$ is the set of complex numbers with real part in $[-1/2,1/2]$ and the boundary of $\cQ$ can be described as follows:
$$
\Bd \cQ=\left\{\frac{ \pm  (1+i\tan \theta)}{2}:\theta\in \left(-\frac{\pi}{2},\frac{\pi}{2}\right) \right\}.
$$
For $d\geq 3$, $\cQ$ is the regular convex polygon with boundary given by the simple closed curve:
$$
\Bd  \cQ=\left\{\frac{   e^{\frac{2k\pi i}{d}}   (1+i\tan \theta)}{2}:\theta\in \left(-\frac{\pi}{d},\frac{\pi}{d}\right ], k\in\{0,\dots,d-1\}\right\}.
$$
Since
\begin{align*}\frac{1}{1-e^{-i\theta}}&=\frac{1}{1-\cos \theta +i \sin \theta}=\frac{1-\cos \theta -i\sin \theta}{2-2\cos\theta} \\&=\frac{1}{2}-\frac{i \sin\theta}{2(1-\cos\theta)}=\frac{1}{2}+\frac{i \sin(\theta+\pi)}{2(1+\cos(\theta+\pi))}\\&
=\frac{ (1+i\tan\frac{ \theta+\pi}{2})}{2},
\end{align*}
we obtain another representation of $\Bd\cQ$:
\begin{equation}\label{eq:BdcQ}
\Bd \cQ=\left\{\begin{array}{ll}
\left\{\frac{ \pm 1 }{ 1-e^{-i\theta}}:\theta\in \left(0,2\pi\right) \right\} & \mathrm{~if~} d=2\\
\left\{\frac{   e^{\frac{2k\pi i}{d}}  }{1-e^{-i\theta}}:\theta\in \left(\pi-\frac{2\pi}{d},\pi+\frac{2\pi}{d}\right ], k\in\{0,\dots,d-1\} \right\},& \mathrm{~if~} d\geq 3
\end{array}\right.
\end{equation}
Plugging ~\eqref{eq:BdcQ} into~\eqref{a:Bdf1f2}  we obtain that
\begin{equation}\label{eq:BdcQ2}
\Bd f_1 (f_2 (\cQ))=\left\{\begin{array}{ll}
\left\{\frac{ 1 }{1-( 1-e^{-i\theta})^2}:\theta\in \left(0,2\pi\right) \right\} & \mathrm{~if~} d=2\\
\left\{\frac{  1  }{1-(1-e^{-i\theta})^d}:\theta\in \left(\pi-\frac{2\pi}{d},\pi+\frac{2\pi}{d}\right ] \right\} & \mathrm{~if~} d\geq 3
\end{array}\right.
\end{equation}
Therefore, define the set
\begin{equation}\label{eq:Sigmad}
\Sigma_{0,d}:=\left\{\begin{array}{ll}
f_1(f_2(\cQ))\cup\{1\} & \mathrm{~if~} d=2\\
f_1(f_2(\cQ)) & \mathrm{~if~} d\geq 3
\end{array}\right.
\end{equation}
Then  we have~\eqref{eq-P-Sigma} and 
\begin{equation}\label{eq:BdcQ3}
\Bd \Sigma_{0,d}=\left\{\begin{array}{ll}
\left\{\frac{ 1 }{1-( 1-e^{-i\theta})^2}:\theta\in \left(0,2\pi\right] \right\} & \mathrm{~if~} d=2\\
\left\{\frac{  1  }{1-(1-e^{-i\theta})^d}:\theta\in \left(\pi-\frac{2\pi}{d},\pi+\frac{2\pi}{d}\right ] \right\} & \mathrm{~if~} d\geq 3
\end{array}\right.
\end{equation}
which can be written as
$$\Bd \Sigma_{0,d}=\left\{ \frac{\e^{i d \theta} }{\e^{i d \theta}-(\e^{i \theta}-1)^d} , \quad \theta \in \left(\pi - \frac{2\pi}{d},\pi + \frac{2\pi}{d}\right]\right\}\enspace.
$$
 for any $d\geq 2$. 
Finally the compactness of $\Sigma_{0,d}$ follows from the compactness of $\cS$, which can be easily seen from the fact that  $\cS\subset  \psi_d\left (\ball (0,1)\right)$ by the definition~\eqref{a:stabled}.
\end{proof}




\begin{proof}[Proof of Theorem~\ref{dSigmaAcc}]
This follows directly from~\eqref{a:specP2},~\eqref{a:stabled} and~\eqref{eq-P-Sigma}.
\end{proof}

\begin{remark}\rm
In Theorem~\ref{dSigmaAcc}, the region $\Sigma_{\epsilon,d}$ assured to be accelerable {\em does not} contain some part of the real interval $[0,1-\epsilon]$ for any $d\geq 3$.
This is consistent with Theorem 2.2.12 of \cite{Nesterovbook}
implying that for a linear recurrent scheme with finite memory
calling the oracle $T$, the geometric convergence rate
cannot be smaller than $1-O(\kappa^{-1/2})$ where $\kappa$ is
a condition number, corresponding to $\epsilon^{-1}$ here.
\end{remark}



\subsection{Robustness of the acceleration}
Note that the parameters $(\alpha_0,\dots, \alpha_{d-1})$ defined in \cref{alpha_i} requires the knowledge of $\epsilon$ thus  of the  exact value of the spectral radius of $P$, which may be a restrictive assumption for practitioners. 
\textcolor{blue}{For example, in the stochastic shortest path problem analyzed in~\cite{bertsekas1991}, we do not know the spectral radii of the substochastic matrices arising in the restricted contracting operator described in Proposition~1 of~\cite{bertsekas1991}.}
In this section we evaluate how the small perturbation of $\epsilon$ will affect the order of convergence of the acceleration scheme. This in particular allows the use of an approximate value of $\epsilon$ to compute the parameters $(\alpha_0,\dots, \alpha_{d-1})$ 
while still achieving an asymptotic convergence rate of order  $1-\Omega(\epsilon^{1/d})$.

For $h\geq 0$, we are looking for the smallest radius $g_\epsilon(h)\geq 0$ such that
$\left(\phi^*_{d}\right)^{-1}(\ball(1-\epsilon,h))\subset \ball(1-\epsilon^{1/d},g_\epsilon(h))$, and we enforce $g_\epsilon(h)\leq \epsilon^{1/d}$ to preserve the acceleration. First we make this analysis for $\psi_d$ (i.e.\ $\epsilon=0$).
\begin{lemma}\label{Tolerance0}
	For $h\geq 0$, the smallest nonnegative real number $g_0(h)$ such that $$\psi_d^{-1}(\ball(1,h)) \subset \ball(1,g_0(h))$$ is 
	\[
	g_0(h)=\frac{h^{1/d}}{(1+h)^{1/d}-h^{1/d}}, \quad \forall h\geq 0 \enspace .
	\]
\end{lemma}
\ifX
\begin{proof}
	For $h=0$, it follows from $\psi_d^{-1}(1)=\{1\}$.
	
	Now let $h>0$, we want to have $\psi_d(\ball(1,g_0(h))^c) \subset \ball(1,h)^c$, i.e.:
	\begin{equation}\label{I1}
	|\lambda|> g_0(h) \Rightarrow |\psi_d(1+\lambda)-1|> h,\enspace \forall \lambda\in \C.
	\end{equation}
	We have $\psi_d(1+\lambda)=\frac{(1+\lambda)^d}{(1+\lambda)^d-\lambda^d}=1+\frac{\lambda^d}{(1+\lambda)^d-\lambda^d}=1+\frac{1}{(1+\frac{1}{\lambda})^d-1}$, then
	$$\left |\psi(1+\lambda)-1\right |> h \Leftrightarrow  \left |\left(1+\frac{1}{\lambda}\right)^d-1\right |<\frac{1}{h},\enspace \forall \lambda\in \C.
	$$
For any $\lambda\in \C$ we know that
	\begin{equation}\label{I2}
	\left |\left(1+\frac{1}{\lambda}\right)^d-1\right |=\left |\sum_{k=1}^{d}\binom{d}{k}\frac{1}{\lambda^k }\right|\leq \sum_{k=1}^{d}\binom{d}{k}\frac{1}{|\lambda|^k}=\left(1+\frac{1}{|\lambda|}\right)^d-1	.
	\end{equation}
	Thus 
	$$ |\lambda|>\frac{h^{1/d}}{(1+h)^{1/d}-h^{1/d}} \Leftrightarrow \left (1+\frac{1}{|\lambda|}\right)^d-1<\frac{1}{h}\Rightarrow \left |\psi(1+\lambda)-1\right |> h,\enspace \forall \lambda\in \C.
	$$
	 This allows to conclude because we have equality in (\ref{I2}) when $\lambda \in \Rset_+$.  
\end{proof}
\fi
\ifX\begin{lemma}\label{a-epsilon}
	For  any $a\in [0,1]$ we have 
	\[
	\left(\phi^*_{d}\right)^{-1}(\ball(1-\epsilon, a\epsilon))\subset \ball(1-\epsilon^{1/d}, a^{1/d}\epsilon^{1/d})\enspace.
	\]
  \end{lemma}
  \fi
\ifX
\begin{proof}
	By the property~\eqref{PropertyPhi} and \Cref{Tolerance0} we deduce that for $h\geq 0$, the smallest radius $g_\epsilon(h)$ such that $\left(\phi^*_{d}\right)^{-1}(\ball(1-\epsilon,h))\subset \ball(1-\epsilon^{1/d},g_\epsilon(h))$ is given by 
	$$g_\epsilon(h)=\frac{(1-\epsilon^{1/d})h^{1/d}}{(1+h-\epsilon)^{1/d}-h^{1/d}}, \quad \forall h\geq 0.
	$$
	Note that
	$$
	(1+h-\epsilon)^{1/d}-h^{1/d} \geq 1-\epsilon^{1/d},\enspace \forall h \leq \epsilon.
	$$
	Hence
	$$
	g_\epsilon (h) \leq h^{1/d}, \enspace \forall h \leq \epsilon.
	$$
	We achieve the proof by taking 
	 $h=a\epsilon$.
\end{proof}
\fi

The following theorem describes a $d$-accelerable region.
\begin{theorem}\label{thm:daccle}
	Let $a \in [0,1[$, if $\spec\nmyp\subset \ball\left(0,\frac{1-\epsilon}{2^d+1}\right)\cup \ball\left(1-\epsilon, a\epsilon\right)$ then with the choice of $\alpha$ specified in~\eqref{alpha_i} we have,
	\[
	\spec Q_{\alpha, d}\subset \ball(0,1-\epsilon^{1/d})\cup \ball(1-\epsilon^{1/d}, a^{1/d}\epsilon^{1/d}) \enspace ,
	\]
	so that the iterates of the $d$A-VI algorithm~\eqref{dAVI} with $\beta=0$  satisfy 
	$$
	\limsup_{k\to\infty}\|x_k-x_*\|^{1/k}
\leq 1-(1-a^{1/d}) \epsilon^{1/d}.
	$$
\end{theorem}
\ifX{\begin{proof}
	By combining \Cref{dSigmaAcc}, \Cref{ball-Sigma} and \Cref{a-epsilon}.
\end{proof}}\fi

\subsection{Application to Markov Decision Processes: Accelerated Policy Iteration}

As an application, we consider the standard discounted Markov decision process (MDP) with
state space $[n]:=\{1,\dots,n\}$, see~\cite{whittle,Bertsekas96} for background. For each state $i$, denote by $\cA(i)$ the set of actions,  $P^a_{i,j}$ the transition probability from state $i$ to state $j$ under action $a\in \cA(i)$,  and $g_i^a$ the 
 reward of choosing action $a\in \cA(i)$ in state $i$. Let $1>\gamma_i>0$, for $i\in[n]$, be state-dependent discount factors. The associated dynamic programming operator $T:\R^n\to\R^n$ is given by:
\begin{align}\label{e-def-mdp}
T_i(x):=\max_{a\in\cA(i)} \gamma_i \sum_{j\in [n]}P^a_{i,j}x_j+ g_i^a,\enspace \forall i\in [n]\,.
\end{align}
We set $\gamma:= \max_{i\in[n]} \gamma_i$. 

\textcolor{blue}{
The value of the discounted problem for this MDP starting from an initial state $i$ is given by:
\[
v_i:=\max_{a_0, a_1, \cdots} \mathbb{E} \big [ g_{X_0}^{a_0}+ \gamma_{X_0} g_{X_1}^{a_1}+ \gamma_{X_0} \gamma_{X_1} g_{X_2}^{a_2} +\cdots \mid X_0 =i \big ] \enspace,
\]
where the maximum is taken over admissible sequences of random actions, and $X_0, X_1, \cdots$ denotes the random sequence of states generated by the actions.}

We are interested in finding the value vector $v\in \R^n$ of this MDP which is a solution of the fixed point
problem $v=T(v)$.
The fixed point exists and is unique since $T$ is a contraction
of constant $\gamma$ in the sup-norm.

A classical approach to solve this problem is to use value
iteration, i.e., to compute the sequence $v^k=T(v^{k-1})$,
which converges to the unique fixed point. It is tempting
to apply directly accelerated value iteration to the non-linear
problem $v=T(v)$. This approach was proposed in~\cite{julien},
and it is experimentally effective on some instances.
However, the convergence proof of accelerated value
iteration uses inherently the affine character
of the operator $T$, and it is not clear whether
general enough convergence conditions can be
given for Markov decision processes. An alternative
approach, which we develop here, is to rely
on {\em policy iteration} instead of value
iteration, which will allow us to
apply the idea of $d$th acceleration to solve MDP, but in an indirect
manner, leading to convergence guarantees.

A {\em policy} is a map $\sigma:[n] \to \cup_{i\in [n]}A(i)$
such that $\sigma(i)\in A(i)$, it represents a state dependent
decision rule. It determines a $0$-player
game, with an affine operator $T^\sigma: \R^n \to \R^n$,
\begin{align}\label{eq:Tsigma}
T^\sigma _i(x):= \gamma_i \sum_{j\in [n]}P^{\sigma(i)}_{i,j}x_j+ g_i^{\sigma(i)},\enspace \forall i\in [n]\,.
\end{align}
For a vector $x\in \R^n$, we define the quantity $\Top(x) =\max_{i\in[n]}x_i$.
We have $\|x\|_\infty=\max(\Top(x),\Top(-x))$. For $x,y\in \R^n$, we write $x\leq y$ to mean that $x_i\leq y_i$ for all $i\in [n]$. We denote by $x_*$ the unique fixed point of the operator $T$ and by $x^{\sigma}$ the unique fixed point of the operator $T^{\sigma}$. We denote by $a^+=\max(a,0)$ the positive part
of a real number. The following lemma presents some classical properties of the operators $T$ and $T^\sigma$ that are useful for our analysis.

\begin{lemma}
	Let $x,y\in \R^n$, $\sigma$ a policy, $e=(1,\cdots,1)\in \R^n$ the unit vector and $a\in \R^{+}$ a nonnegative real number, we have:
	\begin{eqnarray}
	(\Top(T(x)-T(y)))^+\leq \gamma (\Top(x-y))^+ ,\label{eq-T-1}\\
	\|x-x^{\sigma}\|_\infty \leq \frac{1}{1-\gamma} \|x-T^\sigma(x)\|_\infty , \label{eq-T-2}\\
	\|x-x_*\|_\infty \leq \frac{1}{1-\gamma} \|x-T(x)\|_\infty , \label{eq-T-3}\\
	T^\sigma(x+a e)\leq T^\sigma(x)+\gamma a e ,\label{eq-T-4}\\
	x\leq T^\sigma(x)+ a e \Rightarrow x\leq x^\sigma + \frac{a}{1-\gamma} e , \label{eq-T-5}\\
	x\leq T(x)+ a e \Rightarrow x\leq x_* + \frac{a}{1-\gamma} e ,\label{eq-T-6}\\
	x\leq y \Rightarrow T^\sigma (x)\leq T^\sigma(y) . \label{eq-T-7}
	\end{eqnarray}
\end{lemma}
Property~\eqref{eq-T-7} follows from $P_{ij}^a\geq 0$,
whereas~\eqref{eq-T-4}  follows from $\sum_j P_{ij}^a=1$.
Property~\eqref{eq-T-1} means that $T$ is a contraction of
rate $\gamma$ in the nonsymmetric norm $(x,y) \mapsto (\Top(x-y))^+$. To see it we compute for $i\in [n]$, $T_i(x)-T_i(y)=\max_{a}\{\gamma_i \sum_{j\in [n]}P^a_{i,j}x_j+ g_i^a\} - \max_{a}\{\gamma_i \sum_{j\in [n]}P^a_{i,j}y_j+ g_i^a\} \leq (\gamma_i \sum_{j\in [n]}P^{a_x}_{i,j}x_j+ g_i^{a_x}) - (\gamma_i \sum_{j\in [n]}P^{a_x}_{i,j}y_j+ g_i^{a_x})= \gamma_i \sum_{j\in [n]}P^{a_x}_{i,j}(x_j-y_j)$, where $a_x$ is the action that maximizes the expression of $T_i(x)$. 
Property~\eqref{eq-T-2} (and similarly~\eqref{eq-T-3}) comes street forward from $x^{\sigma}$ being a fixed point of $T^{\sigma}$ and the later being a $\gamma$-contraction in the sup norm.

To obtain property~\eqref{eq-T-5} (and similarly~\eqref{eq-T-6}), we apply $k$ times the operator $T^\sigma$ to both sides of the initial inequality and we use the properties~\eqref{eq-T-7} and~\eqref{eq-T-4} to obtain that $x\leq (T^\sigma)^{k+1}(x)+ a \sum_{i=0}^{k} \gamma^i e$ and finely since $T^\sigma$ is a strict contraction, we know that when $k$ goes to infinity, $(T^\sigma)^{k+1}(x)$ converges to the fixed point $x^\sigma$.

Policy iteration computes
a succession of policies $\sigma^1,\sigma^2,\dots$.
At each stage, it solves a $0$-player fixed point problem,
finding a vector $v^k$ such that
$v^k = T^{\sigma^k}(v^k)$. Then, the vector $v^k$
is used to determine the new policy, by considering
the maximizing actions in the expression of $T(v^k)$.
When policy iteration is implemented in exact arithmetics, for a fixed $\gamma<1$, the number of iterations is
strongly polynomial~\cite{ye2010simplex}.
Moreover, on ordinary instances, the number of iterations is often
of a few units. Hence, the bottleneck, preventing to apply policy iteration
to large scale Markov decision problems, is generally
the solution of the affine problem
$v^k = T^{\sigma^k}(v^k)$: algebraic methods, based
on LU-factorization, are not adapted to large scale sparse instances,
whereas standard iterative methods can be slow, since the contraction
rate $\gamma$ is typically close to $1$.
To address this difficulty, we present a version of policy iteration in which
at each stage, $v^k$ is computed by the $d$th accelerated scheme.



We consider the {\em Accelerated Policy Iteration of degree $d$ ($d$A-PI)}
presented in~\Cref{d-API}.
\begin{algorithm}[htbp]
	\begin{algorithmic}[1]
	\State Fix a target accuracy $\delta$ for value determination and $\delta'$ for policy improvement.
    \State Initialization: select a starting policy $\sigma^{0}$,
	and set the initial values $x_{-1,0}=x_{-1,1}=\cdots=x_{-1,d-2}=y_{-1,d-2}=0$

	\For{$k=0,1,\cdots$} the following:

	\State (Accelerated value determination): Run the $d$A-VI \cref{dAVI} on the operator $T^{\sigma^k}$ until having a residual smaller that $\delta$: so first we initialize $x_{k,0},x_{k,1},\cdots,x_{k,d-2}$ by the last $d-1$ values of the sequence $(x_{k-1,l})_l$ and $y_{k,d-2}$ by the last value of the sequence $(y_{k-1,l})_l$, and for $l=d-2,\cdots$, we do the iterations of \cref{dAVI}:
	\begin{subequations}
		\begin{align}
		x_{k,l+1}&=(1-\beta)y_{k,l}+ \beta T(y_{k,l}) \enspace ,
		\\
		y_{k,l+1}&=(1+\alpha_{d-2}+\cdots+\alpha_0)x_{k,l+1} -\alpha_{d-2} x_{k,l} -\cdots -\alpha_0 x_{k,l-d+2} \enspace , 
		\end{align}
	\end{subequations}
	until $\|y_{k,l}-T^{\sigma^k}(y_{k,l})\|_{\infty}\leq \delta$. We denote the final $y_{k,l}$ by $y_k$.
	
    \State (Policy improvement). We determine a policy $\sigma^{k+1}$ such that $\|T(y_k)-T^{\sigma^{k+1}}(y_k)\|_{\infty}\leq \delta'$,
    and for each $i \in [n]$, we choose $\sigma^{k+1}(i)=\sigma^k(i)$ whenever possible. 
    \EndFor
    \end{algorithmic}
    \caption{Accelerated Policy Iteration of degree $d$ ($d$A-PI).}\label{d-API}
\end{algorithm}
Using classical estimates on approximate value iteration, see~\cite{Bertsekas96,bertsekas11,scherrer}, we get the
following convergence result.
\begin{proposition}\label{prop-dAPI}
  Suppose that 
  for any policy $\sigma$, $\spec P^{\sigma} \subset \Sigma_{\epsilon,d}\cup \{1-\epsilon\} $,
  and that we choose $\alpha=(\alpha_0,\cdots,\alpha_{d-2})$ as in \cref{alpha_i}. 
  Each iteration $k$ of the $d$A-PI algorithm terminates, 
  and we have :
  \begin{equation}\label{eq-limsup}
  \limsup_{k\to\infty}\|y_k-x_*\|_\infty \leq \frac{(1+\gamma)\delta+\delta'}{(1-\gamma)^2} \enspace .
  \end{equation}
  Moreover, if $\sigma^{k+1}=\sigma^{k}$ for some $k$, then $\|y_k-x_*\|_\infty \leq \frac{\delta+\delta'}{1-\gamma}$.
  \end{proposition}
\begin{proof}
	The termination of each iteration $k$ 
	comes from \cref{dSigmaAcc}.
	For each $k$, we have from the algorithm $\|y_{k}-T^{\sigma^k}(y_{k})\|_{\infty}\leq \delta$, then $y_k\leq T^{\sigma^k}(y_{k}) +\delta e \leq T(y_k) + \delta e$. Then by~\eqref{eq-T-6} we deduce that $y_k\leq x_*+\frac{\delta}{1-\gamma} e$. Therefore $\Top(y_k-x_*)\leq \frac{\delta}{1-\gamma}\leq \frac{(1+\gamma)\delta+\delta'}{(1-\gamma)^2}$.
	
	We have $y_k \leq T(y_k) + \delta e \leq T^{\sigma^{k+1}}(y_{k}) +(\delta+\delta') e$, then by~\eqref{eq-T-5} we get $y_k\leq x^{\sigma^{k+1}}+\frac{\delta+\delta'}{1-\gamma} e$. By~\eqref{eq-T-2} and the algorithm, we have $\|y_{k+1}-x^{\sigma^{k+1}}\|_\infty \leq \|y_{k+1}-T^{\sigma^{k+1}}(y_{k+1})\|_\infty/(1-\gamma)
	\leq \delta/(1-\gamma)$, then $x^{\sigma^{k+1}}\leq y_{k+1} + \frac{\delta}{1-\gamma} e$. We deduce that $y_k\leq y_{k+1} + \frac{2\delta+\delta'}{1-\gamma} e$. We apply $T^{\sigma^{k+1}}$ to both sides of this inequality and use~\eqref{eq-T-7} and~\eqref{eq-T-4} to get that $T(y_k)\leq
	T^{\sigma^{k+1}}(y_k) +\delta' e\leq T^{\sigma^{k+1}} (y_{k+1} + \frac{2\delta+\delta'}{1-\gamma} e) +\delta' e
	\leq T^{\sigma^{k+1}}(y_{k+1})+ \frac{(2\delta+\delta')\gamma}{1-\gamma} e +\delta' e
	\leq y_{k+1} +\delta e +\frac{(2\delta+\delta')\gamma}{1-\gamma} e +\delta' e
	= y_{k+1} +\frac{(1+\gamma)\delta+\delta'}{1-\gamma} e$. 
	Therefore, $x_*-y_{k+1}\leq x_*-T(y_k) +\frac{(1+\gamma)\delta+\delta'}{1-\gamma} e = T(x_*)-T(y_k) +\frac{(1+\gamma)\delta+\delta'}{1-\gamma} e$. 
	Then $(\Top(x_*-y_{k+1}))^{+}\leq (\Top(T(x_*)-T(y_k)))^{+} +\frac{(1+\gamma)\delta+\delta'}{1-\gamma}$, and by using~\eqref{eq-T-1} we deduce that $(\Top(x_*-y_{k+1}))^{+}\leq \gamma (\Top(x_*-y_k))^{+} +\frac{(1+\gamma)\delta+\delta'}{1-\gamma}$. By iterating this inequality, we deduce that for iteration $k$, $(\Top(x_*-y_{k}))^{+}\leq \gamma^k (\Top(x_*-y_0))^{+} +\frac{(1+\gamma)\delta+\delta'}{1-\gamma}\sum_{i=0}^{k-1} \gamma^i$. Therefore, $\limsup_{k\to\infty} (\Top(x_*-y_{k}))^{+} \leq \frac{(1+\gamma)\delta+\delta'}{(1-\gamma)^2}$, and by using $\|x_*-y_{k}\|_\infty=\max(\Top(y_{k}-x_*),(\Top(x_*-y_{k}))^{+})$ we end the proof of~\eqref{eq-limsup}.
	
	
	Now, if $\sigma^{k+1}=\sigma^{k}$ for some $k$, then $\|T(y_k)-T^{\sigma^{k}}(y_k)\|_\infty\leq \delta'$ and we know that $\|y_{k}-T^{\sigma^k}(y_{k})\|_{\infty}\leq \delta$, then $\|y_{k}-T(y_{k})\|_{\infty}\leq \delta +\delta'$. Therefore by~\eqref{eq-T-3}, we get $\|y_{k}-x_*\|_{\infty}\leq \frac{\delta+\delta'}{1-\gamma}$.
\end{proof}

\textcolor{blue}{
  \begin{remark}\rm\label{rem-scherrer}
    \Cref{prop-dAPI} should be compared with~\cite[Prop.~6.2]{Bertsekas96}
    and Remark~5 and Eqn~22 of~\cite{scherrer},
    which bound the same limsup by an expression of the form $(\delta'+2\gamma \epsilon)/(1-\gamma)^2$, where $\epsilon$ is a upper bound of $\|y_k-x^{\sigma^k}\|_\infty$. Here, $\epsilon$ is replaced by $\delta$, which is an upper
    bound of the residual $\|y_k -T^{\sigma^k}(y_k)\|_\infty$.
  \end{remark}
  \begin{remark}\rm\label{rem-complexity-pi}
    \Cref{prop-dAPI} is only an asymptotic result. In contrast,
    when policy iteration is implemented exactly, the value vector
    $v^k$ associated to the $k$th policy that is selected
    satisfies $\|v_k -x_*\|_\infty
    \leq \gamma^k \|v_0-x_*\|_\infty$, see~Lemma~6.5 of~\cite{Hansen2013}.
    \end{remark}
\begin{remark}\rm\label{rem-num-stability}
  Since accelerated value iteration, and so, accelerated
  policy iteration, are implemented with a fixed precision
  arithmetics, one may wonder whether acceleration leads
  to numerical unstabilities. In the numerical experiments
  which follows, no such unstabilities
  were observed for the relevant values $d=2,4$ considered
  here. We verified the validity of the approximate solutions
  that we obtained using
  the inequality
~\eqref{eq-T-3}. Indeed,
  the residual $\|y_k  - T(y_k)\|_\infty$, where $y_k$
  is the approximate solution gotten at the final iteration of the algorithm,
  can be evaluated in an accurate way (with a precision close to the machine precision) using only the last value $y_k$. So, if this residual is small, by the inequality~\eqref{eq-T-3}, we can certify that $\|y_k - x^*\|_\infty$ is also small,
  so that we have a valid approximate solution.
  In all the experiments of~\Cref{sec-numerical}, the algorithms are stopped with a residual of $<10^{-10}$,
  and $1-\gamma$ is $\geq 10^{-4}$, so, it is guaranteed that the true
  solution is approximated with a precision $<10^{-6}$.
\end{remark}}

\section{Numerical results}\label{sec-numerical}
In this section, we show the numerical performance of the proposed $d$A-VI and $d$A-PI  with $d=2$ and $d=4$.
The acceleration parameters in all the examples follow~\cref{alpha_i}; the parameter $\alpha=\frac{1-\sqrt{\epsilon}}{1+\sqrt{\epsilon}}$ for accelerations of degree $2$, and the parameters
$ \alpha_0=\frac{(1-\epsilon^{1/4})^4}{1-\epsilon},
 \alpha_1= \frac{-4(1-\epsilon^{1/4})^3}{1-\epsilon}$ and $\alpha_2=\frac{6(1-\epsilon^{1/4})^2}{1-\epsilon}$
for  accelerations of degree $4$.

In all the examples below, we do the policy improvement at each iteration $k$ of the $d$A-PI algorithm in an exact way by taking, for each $i\in [n]$, $\sigma^{k+1}(i)\in [m]$ to be a value achieving the maximum when evaluating~\eqref{e-def-mdp} at $x=y_k$, i.e. $\delta'=0$, and we let the accuracy of the value determination to be $\delta=10^{-10}$.



\subsection{Markov decision processes with random matrices}\label{subsec-random}
We consider the discounted MDP  model of~\eqref{e-def-mdp}.
We take a damping parameter $\beta=1$ in what follows.

The instances used in~\Cref{myfigure2,myfigure3,myfigure4,myfigure41,myfigure42,myfigure5} are generated in the following way.
We fix two integers $n$ and $m$. For each $i\in [n]$, we take $\cA(i)=[m]$ and randomly generate a probability vector $p_i^a=(P^a_{i,1},\dots, P^a_{i,n})$  as follows:
$P^a_{i,j}=\frac{X^a_{i,j}}{X^a_{i,1}+\dots+X^a_{i,n}}$,
where the $X^a_{i,j}$ are independent Bernoulli random variables of mean $p\in (0,1)$.  The discount factors $\gamma_i$ are randomly chosen in the interval $[1-2\epsilon,1-\epsilon]$, independently for each $i\in [n]$. 

Let $\lambda_1,\dots \lambda_n$ be the eigenvalues of $\sqrt{n}P$.
It is shown in \cite{Bordenave08} that the counting probability measure 
$\frac{\delta_{\lambda_1}+\cdots+\delta_{\lambda_n}}{n}$,
converges weakly as $n\rightarrow \infty$ to the uniform law on the disk
$\{z\in \C:|z|\leq \sqrt{(1-p)/p}\}$. Moreover, Theorem 1.2, {\em ibid.}
shows that the second modulus of an eigenvalue of $P$ is
of order $1/\sqrt{n}$.
This explains the shape of the spectrum shown on the figures \Cref{myfigure2,myfigure3,myfigure4}, 
and explains also, along with~\eqref{ball-Sigma}, why the accelerated schemes of order $4$ work in the large scale example of~\Cref{myfigure5} where we take $p=0.0025$ with $n=10^5$.


\ifX{
In~\Cref{myfigure2,}, we consider an
instance where the matrices are randomly generated as above with $n=30$, $m=10$ and $p=0.2$.
In subfigure~\ref{fig:subfig22}, we display the spectrum of one matrix $P^\sigma_\gamma:=(\gamma_i P^\sigma_{ij})_{ij}$.
One can notice that this spectrum presents eigenvalues that are outside the simply and multiply accelerable regions delimited respectively by $\Gamma_\epsilon$ and $\Gamma_{\epsilon,4}$ (see~\Cref{dSigmaAcc}). Therefore, the accelerated policy iteration algorithms ($d$A-PI) cannot be applied for this instance.
In accordance with that, the subfigure~\ref{fig:subfig21} shows that the accelerated value iteration algorithms $2$A-VI and $4$A-VI do not converge.

In~\Cref{myfigure3}, we consider an instance with $n=100$, $m=10$ and $p=0.2$. The subfigure~\ref{fig:subfig32} shows that the spectrum of the random matrices in this case is located in the simply accelerable region delimited by $\Gamma_{\epsilon}$, but it is not included in the $4$-accelerable region $\Gamma_{\epsilon,4}$. Therefore, we can apply the $2$A-PI algorithm but not the $4$A-PI in this case. The subfigure~\ref{fig:subfig31} shows that the simply accelerated schemes $2$A-PI and $2$A-VI has significantly better performances than value iteration algorithm. It shows also as expected that the acceleration of order $4$ does not converge. 

In~\Cref{myfigure4}, we consider an instance with $n=1500$, $m=10$ and $p=0.2$. The subfigure~\ref{fig:subfig42} shows that the spectrum of the random matrices in this case is located inside the accelerable regions of order $2$ and $4$ delimited respectively by $\Gamma_{\epsilon}$ and $\Gamma_{\epsilon,4}$. Therefore, we can apply both $2$A-PI and $4$A-PI in this case. The subfigure~\ref{fig:subfig41} shows that all the accelerated schemes converge in this case and that the multi-accelerated schemes have better performances than the simply accelerated ones.

}\fi

In~\Cref{myfigure41}, we consider an instance with $n=4000$, $m=10$, $p=0.1$. In this example we take $\epsilon=10^{-2}$ to allow the Value Iteration algorithm to have a visible improvement. 

In~\Cref{myfigure42}, we consider an instance with $n=4\times 10^4$, $m=10$ and the matrices used are sparse with a parameter $p=0.005$.  

\textcolor{blue}{
We observe that the classical Policy Iteration (PI) algorithm~\cite{howard1960dynamic,puterman2014markov}, using LU decomposition to solve the linear $0$-player problem at each iteration, is way more faster than our iterative algorithms ($d$A-PI and $d$A-VI) in the case of small matrices like in~\Cref{myfigure2,myfigure3}, but as the size of the matrices gets bigger our iterative algorithms become more competitive like in~\Cref{myfigure4,myfigure41}, and even way faster than Policy Iteration like in~\Cref{myfigure42}.
}

The~\Cref{myfigure5} represents a large scale analogue to the previous examples where the number of states is $n=10^5$,
and the matrices used are sparse with a parameter $p=0.0025$. 
For this example, the classical Policy Iteration algorithm
cannot be used because of memory saturation.
However, the $d$A-PI algorithms~\ref{d-API} that we propose, with $d=2$ and $d=4$ here, work in this case and show significantly better performances than the classical Value Iteration algorithm. 
The $d$A-VI algorithms also show competitive performances in comparison with $d$A-PI algorithms. However, we expect in general that $d$A-PI becomes more competitive than $d$A-VI when the number of actions $m$ is large, because the number of policies visited grow slowly
with the number of actions (in the discounted case,
a worst case almost linear bound for this number
is given in~\cite{scherrer2013improved}, based on~\cite{ye2010simplex},
the convergence being generally faster on typical instances).

In particular, for all the examples in~\Cref{myfigure3,myfigure4,myfigure41,myfigure42,myfigure5}, we notice that both $2$A-PI and $4$A-PI stop only after $k\leq 5$ iterations over policies because each one of them finds a policy $\sigma^{k+1}$ equal to $\sigma^{k}$. The same phenomenon occurs in the second application shown in the next section (see~\Cref{myfigure6,myfigure7} below).




\ifX{\begin{figure}[htbp]
	\centering
	\subfigure[PI, VI, $2$A-VI and $4$A-VI]{
		\includegraphics[scale =0.28] {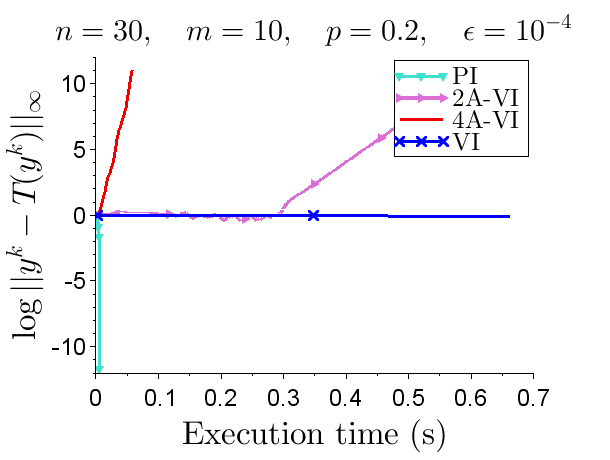}
		\label{fig:subfig21}
	}
	\subfigure
	[The spectrum of one matrix $P^\sigma_\gamma$ where  $\epsilon=10^{-4}$.]{
		\includegraphics[scale =0.28] {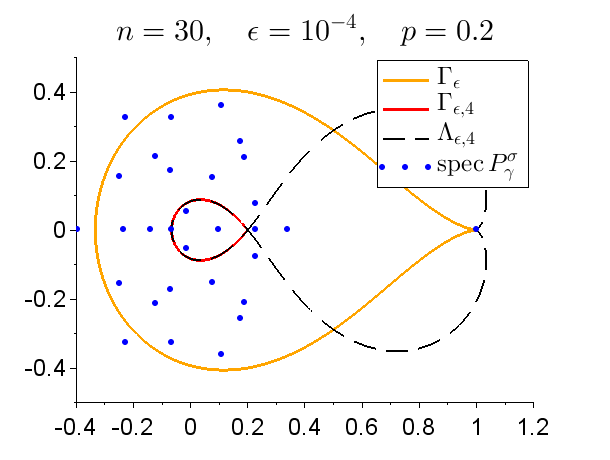}
		\label{fig:subfig22}
	}
	\caption{Markov Decision Process with random Markov matrices of size $n=30$ and with $m=10$ actions at each state.}
	\label{myfigure2}
\end{figure}

\begin{figure}[htbp]
	\centering
	\subfigure[PI, $2$A-PI, VI, $2$A-VI and $4$A-VI]{
		\includegraphics[scale =0.28] {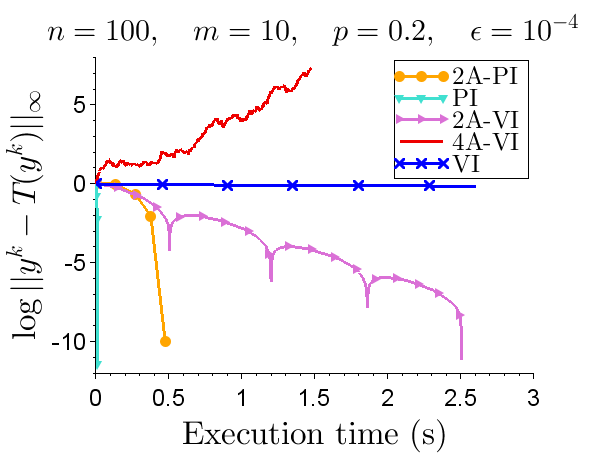}
		\label{fig:subfig31}
	}
	\subfigure
	[The spectrum of one matrix $P^\sigma_\gamma$ where  $\epsilon=10^{-4}$.]{
		\includegraphics[scale =0.28] {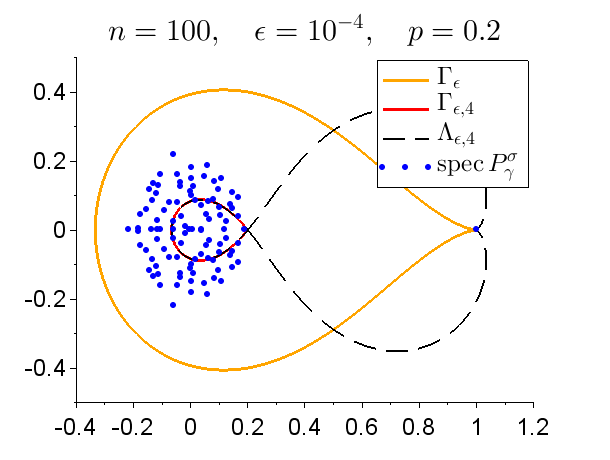}
		\label{fig:subfig32}
	}
	\caption{Markov Decision Process with random Markov matrices of size $n=100$ and with $m=10$ actions at each state.}
	\label{myfigure3}
\end{figure}
}\fi
\begin{figure}[htbp]
	\centering
	\subfigure[PI, $2$A-PI, $4$A-PI, VI, $2$A-VI and $4$A-VI]{
		\includegraphics[scale =0.28] {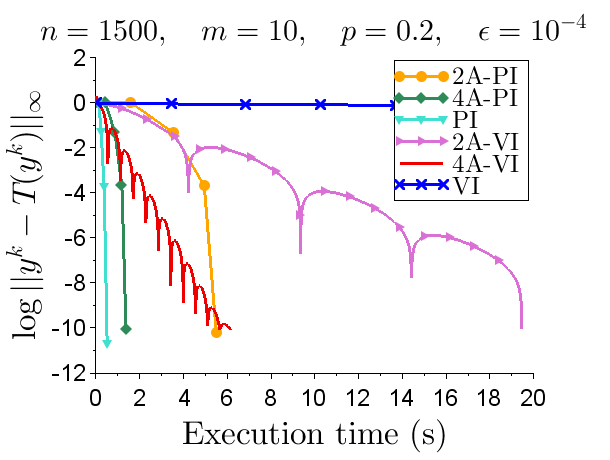}
		\label{fig:subfig41}
	}
	\subfigure
	[The spectrum of one matrix $P^\sigma_\gamma$ where  $\epsilon=10^{-4}$.]{
		\includegraphics[scale =0.28] {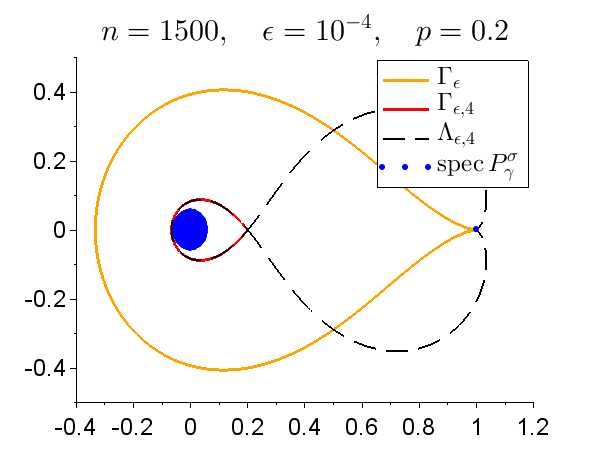}
		\label{fig:subfig42}
	}
	\caption{Markov Decision Process with random Markov matrices of size $n=1500$ and with $m=10$ actions at each state.}
	\label{myfigure4}
\end{figure}

\begin{figure}[htbp]
	\centering
	\subfigure[PI, $2$A-PI, $4$A-PI, VI, $2$A-VI and $4$A-VI]{
		\includegraphics[scale =0.28] {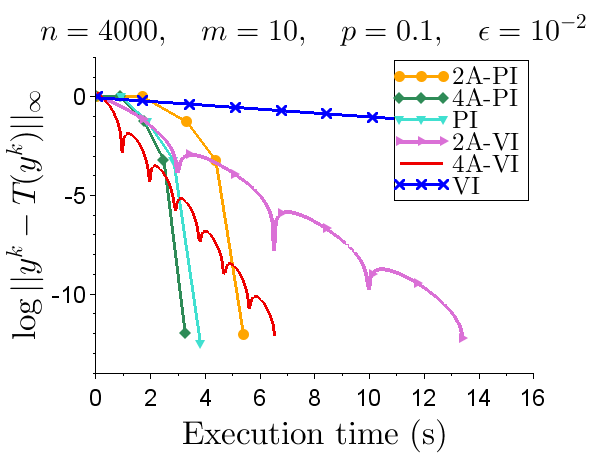}
		\label{fig:subfig411old}
	}
	\subfigure
	[The spectrum of one matrix $P^\sigma_\gamma$ where  $\epsilon=10^{-2}$.]{
		\includegraphics[scale =0.28] {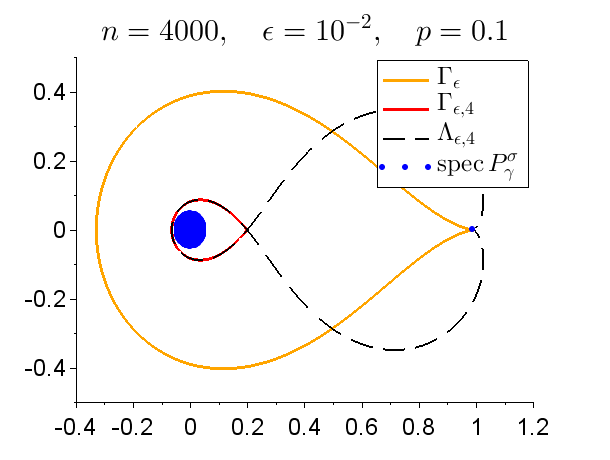}
	}
	\caption{Markov Decision Process with random Markov matrices of size $n=4000$ and with $m=10$ actions at each state.}
	\label{myfigure41}
\end{figure}

\begin{figure}[htbp]
	\centering
	\subfigure[PI, $2$A-PI, $4$A-PI, VI, $2$A-VI and $4$A-VI]{
		\includegraphics[scale =0.28] {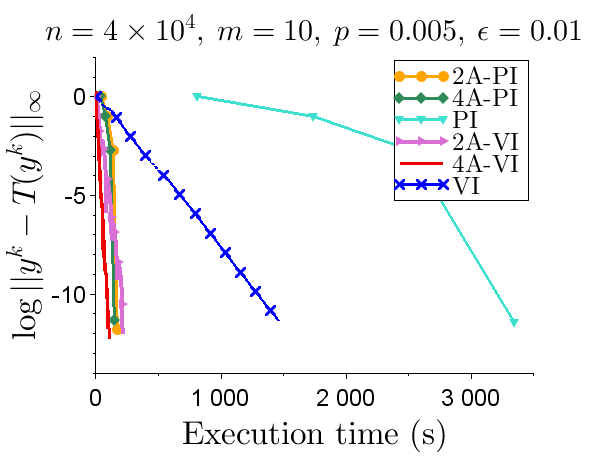}
		\label{fig:subfig411}
	}
	\subfigure[Zoom on~\Cref{fig:subfig411}.]{
		\includegraphics[scale =0.28] {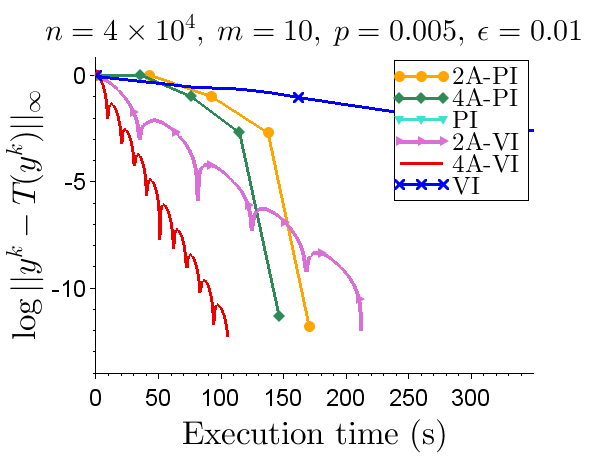}
		\label{fig:subfig412}
	}	
	\caption{Markov Decision Process with random Markov matrices of size $n=4\times 10^4$ and with $m=10$ actions at each state.}
	\label{myfigure42}
\end{figure}

\begin{figure}[htbp]
	\centering
	\includegraphics[scale =0.28] {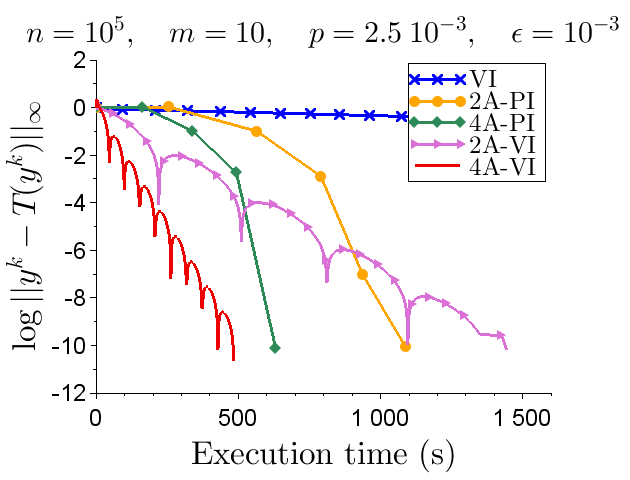} 
	
	\caption{Markov Decision Process with random Markov matrices of size $n=10^5$ and with $m=10$ actions at each state.
}
	\label{myfigure5}
\end{figure}

\subsection{Hamilton-Jacobi-Bellman PDE}\label{subsec-HJB}
We now apply the accelerated schemes to solve a Hamilton-Jacobi-Bellman (HJB) equation arising from a controlled diffusion problem with a small drift. 

\subsubsection{Description of the problem}

We consider an HJB equation in dimension $p\geq 1$, where $v$ is a real-valued function defined on the torus
$\R^p/\Z^p$, identified to $[0,1]^p$, assuming a cyclic
boundary condition:
\begin{equation}\label{eq-pde}
\max_{a\in[m]}\bigg(\frac{1}{2}\sum_{i=1}^{p} \sigma_{i}^2 
\frac{\partial^2v}{\partial x_i^2} (x)
+\sum_{i=1}^{p} g_i(a,x) 
\frac{\partial v}{\partial x_i} (x)
-\lambda v(x) +r(a,x)\bigg)=0  , \enspace x \in [0,1]^p \enspace ,
\end{equation}
where $[m]=\{1,\dots,m\}$ is the set of actions, the scalar $\sigma_i>0$ represents the volatility in direction $i$, $g_i:[m]\times [0,1]^p \mapsto \R$ represents the drift in direction $i$ that depends on the action $a$ and the state $x$, $\lambda>0$ is a dissipation parameter and $r:[m]\times [0,1]^p \mapsto \R$ is the function of rewards.

\textcolor{blue}{
The function $v$ is given by
\[
v(x) = \sup_{a(\cdot)} \mathbb{E} \big[ \int_0^{\infty} \exp(-\lambda t) r(a(t), X_t) dt \mid X_0 = x \big] \enspace,
\]
with $dX_t=g(a(t), X_t) dt +\diag(\sigma) dW_t$, where $W_t$ is the standard Brownian motion on $\R^p$, $\diag(\sigma)$ is the diagonal matrix with entries $(\sigma_i)_{i\in [p]}$ and the supremum is taken over progressively measurable processes $a(t)$ with respect to the filtration of the Brownian motion $W_t$, see~\cite{fleming2006} for background. 
}

For $x=(x_1,\cdots,x_p)$ and $i\in [p]$, we denote by $x_{\neq i}$ the $p-1$ entries of $x$ that are different from $i$. For a scalar $g\in \R$, we denote $g^{+}=\max (g,0)$ and $g^{-}=\max (-g,0)$.

We use a uniform grid $\Omega=\{h,2h,\dots,Nh\}^p$ to discretize the space $[0,1]^p$, where $N$ is a positive integer and $h=1/N$. 
\textcolor{blue}{An upwind finite difference discretization of the HJB equation~\eqref{eq-pde} leads to
}

\begin{multline}\label{HJBdiscret}
\max_{a\in[m]}\bigg(\frac{1}{2}\sum_{i=1}^{p} \sigma_{i}^2 \frac{v(x_{\neq i},x_i+h)+v(x_{\neq i},x_i-h)-2 v(x)}{h^2} \\
+\sum_{i=1}^{p} g_i(a,x)^{+} \frac{v(x_{\neq i},x_i+h)-v(x)}{h} \\
+\sum_{i=1}^{p} g_i(a,x)^{-} \frac{v(x_{\neq i},x_i-h)-v(x)}{h} -\lambda v(x) +r(a,x)\bigg) = 0, \quad x\in \Omega \enspace .
\end{multline}
\textcolor{blue}{This equation reduces to a finite dimensional dynamic programming equation of the form $V=T(V)$, with $T$ as in~\eqref{e-def-mdp}, see~\cite{kushner2001} for background. We next recall this transformation, in order to apply our method. 
}

We consider a discrete vector $V=(V_k)_{k\in [N]^p} \in \R^{N^p}$ such that for each index $k=(k_1,\cdots,k_p)\in [N]^p$, the $k$th entry of $V$ is $V_k=v(hk)$.

The equation \cref{HJBdiscret} can be rewritten in the following matrix form:
\begin{equation}\label{eq-Atau}
\max_{\tau \in [m]^{N^p}}(A_h^{\tau}V+r^{\tau})=0
\end{equation}
such that for a given policy $\tau: [N]^p \to [m]$, the matrix $A_h^{\tau}\in \R^{N^p\times N^p}$ has the $k$th row $(A_h)_{k\cdot}^{\tau(k)}$, $k\in [N]^p$, that represents the equation \cref{HJBdiscret} for $x=hk\in \Omega$ and 
$a=\tau(k)\in [m]$, and where the vector $r^{\tau}$ has the $k$th entry $r^{\tau(k)}_k= r(\tau(k),hk)$ 
.

We can easily see from~\cref{HJBdiscret} that the diagonal entries of each matrix $A_h^{\tau}$ are negative, while all the other entries are nonnegative, and this is due to the distinction of the positive and negative parts of the functions $g_i$ that we did
. We transform the problem \cref{eq-Atau} by introducing for each policy $\tau$ the matrix $P_h^{\tau}=\I+ch^2 A_h^{\tau}$, where $c$ is a positive scalar that allows all the matrices $P_h^{\tau}$ to have nonnegative entries. The following lemma shows how such a scalar can be chosen.

\begin{lemma}
	If $c\leq c_0:=1/(\sum_{i=1}^{p} \sigma_{i}^2+h \max_{a\in[m], k\in [N]^p}\sum_{i=1}^{p} |g_i(a,hk)|+h^2\lambda)$, then for each policy $\tau$, all the entries of the matrix $P_h^{\tau}$ are nonnegative. 
	
	Moreover, we have $P_h^{\tau} e = (1-ch^2\lambda) e$, where $e=(1,\cdots,1)\in \R^{N^p}$, and then $\spec P_h^{\tau} \subset \ball(0,1-\epsilon)$ with $\epsilon=ch^2\lambda$.
\end{lemma}
\begin{proof}
By construction of $P_h^{\tau}$, all its non-diagonal entries are nonnegative. 
	 
For $k\in [N]^p$, we can see from equation \cref{HJBdiscret} that 
\[
(A_h^{\tau})_{kk}=-\sum_{i=1}^{p} \sigma_{i}^2/h^2-\sum_{i=1}^{p} (g_i(\tau(k),hk)^{+}+g_i(\tau(k),hk)^{-})/h-\lambda .
\]
Therefore $(P_h^{\tau})_{kk}=1-ch^2\lambda-c\sum_{i=1}^{p} \sigma_{i}^2 -ch \sum_{i=1}^{p} |g_i(\tau(k),hk)|$ . Then if $c\leq c_0$, all the diagonal entries of $P_h^{\tau}$ are also nonnegative.

The property $P_h^{\tau} e = (1-ch^2\lambda) e$ can be easily seen when we take $v$ equal to the constant vector $e$ in the equation \cref{HJBdiscret}, and since all the entries of $P_h^{\tau}$ are nonnegative, we deduce that its spectral radius is $1-ch^2\lambda$ which ends the proof of the lemma.
\end{proof}

\begin{remark}\rm
We notice that the parameter $c$ used in the definition of $P_h^{\tau}$ plays the role of a Krasnosel'ski\u\i-Mann damping (see~\eqref{a:yk}).
So if we divide $c$ by $2$, i.e. we take $c\leq c_0/2$, this ensures that all the eigenvalues of the matrix $P_h^{\tau}$ has a real part in the interval $[0,1-\epsilon]$.
\end{remark}

Now, we can write the equation \cref{HJBdiscret}, as a fixed point problem that represents a $1$-player game:
\begin{equation}\label{eq-fixpoint}
T(V)=V
\end{equation}
where
$$
T(V)=\max_{\tau\in[m]^n}(P_h^{\tau}V+r_h^{\tau}) \enspace .
$$
with $r_h^{\tau}=ch^2 r^{\tau}$.
\subsubsection{Study of the eigenvalues for uncontrolled PDE with uniform drifts}

We will restrict the study of the eigenvalues of the matrices defining the problem \cref{HJBdiscret}, to the uncontrolled case where $m=1$. We have only one matrix $A_{h}$, and $P_h=\I+ch^2A_h$. We suppose also that the drift coefficients $g_i\in \R$ does not depend on the state $x$. Under this framework we have the following lemma that gives an explicit expression of the eigenvalues of $P_h$.

\begin{lemma}\label{lem-eigenval-pde1}

The $N^p$ eigenvalues of the matrix $P_h$ are given as follows for each $k\in[N]^p$:
\begin{equation*}
\eta(k)= 1 - c\sum_{j=1}^{p} \sigma_{j}^2(1-\cos(2\pi k_j h))-c\lambda h^2
+2i c h\sum_{j=1}^{p} \sin(\pi k_j h) (g_j^{+}\e^{i \pi k_j h}-g_j^{-}\e^{-i \pi k_j h}) \enspace .
\end{equation*}
\end{lemma}
\begin{proof}

	For a given $k\in[N]^p$, we define the vector $V\in \R^{[N]^p}$ which $l\in [N]^p$ entry is given by
$V_l=\e^{2 i \pi h\sum_{j=1}^{p}k_jl_j }$. From \cref{HJBdiscret}, we can verify that
	\begin{multline}
	(A_h V)_l=V_l  \bigg(\frac{1}{2}\sum_{j=1}^{p} \sigma_j^2 \frac{\e^{2i\pi h k_j}-2+\e^{-2i\pi h k_j}}{h^2} \\
	+\sum_{j=1}^{p}\bigg( g_j^{+} \frac{\e^{2i\pi h k_j}-1}{h}
	-g_j^{-} \frac{1-\e^{-2i\pi h k_j}}{h} \bigg) -\lambda
	\bigg) .
	\end{multline}
	Then this shows that 
	\[\mu(k):= \sum_{j=1}^{p} \sigma_{j}^2(\cos(2\pi k_j h)-1)/h^2-\lambda 
	+2i \sum_{j=1}^{p} \frac{\sin(\pi k_j h)}{h}(g_j^{+}\e^{i \pi k_j h}+g_j^{-}\e^{-i \pi k_j h})
	\]
	is an eigenvalue of the matrix $A_h$, and this allows to find all the $N^p$ eigenvalues of $A_h$ and therefore those of $P_h$ also.
\end{proof}

\begin{lemma}\label{lem-eigenval-pde2}
The eigenvalues of the matrix $P_h$ satisfy the following inequality:
\begin{equation*}
|\operatorname{Im}(\eta(k))|\leq \bigg(\sum_{j=1}^{p} \frac{2g_j^2}{\lambda \sigma_j^2}\bigg)^{\frac{1}{2}} \sqrt{\epsilon \big(1-\epsilon - \operatorname{Re}(\eta(k))\big)}, \quad  k\in[N]^p \enspace .
\end{equation*}
\end{lemma}
\begin{proof}
	From \Cref{lem-eigenval-pde1} and using that $g_j^{+}-g_j^{-}=g_j$, $g_j^{+}+g_j^{-}=|g_j|$ and $\epsilon = c\lambda h^2$, we deduce that the real and imaginary parts of the eigenvalue $\eta(k)$ are:
	\[
	\operatorname{Im}(\eta(k))=2ch \sum_{j=1}^{p} g_j\sin(\pi k_j h) \cos(\pi k_j h) \enspace ,
	\]
	\[
	\operatorname{Re}(\eta(k))= 1 - \epsilon -2c \sum_{j=1}^{p}(\sigma_j^2 +h|g_j|) (\sin(\pi k_j h))^2 \enspace .
	\]
	By using Cauchy-Schwartz inequality, we have:
	\[
	\sum_{j=1}^{p} |g_j\sin(\pi k_j h)| \leq \big(\sum_{j=1}^{p} \frac{g_j^2}{\sigma_{j}^2} \big)^{\frac{1}{2}} 
	\big(\sum_{j=1}^{p} \sigma_j^2 (\sin(\pi k_j h))^2 \big)^{\frac{1}{2}} \enspace .
	\]
	and this implies the desired inequality.
\end{proof}

Recall that if the spectrum of a matrix is in the region
$\Sigma_\epsilon(r)$ with the choice of $r$ shown in~\Cref{fig:sigmarb},
the $2$A-VI algorithm, applied to this matrix, converges with an asymptotic rate
$1-\sqrt{\epsilon}/2$ (see~\Cref{RmkSigmastronger}).
If follows from \Cref{lem-eigenval-pde2} that for a fixed value of $\epsilon$,
if the drift coefficients $g_i$ are
sufficiently small, the spectrum of the matrix
$P_h$ lies in a small neighborhood of the real segment $[0,1-\epsilon]$,
and so it satisfies the condition for acceleration
with the latter asymptotic rate. Moreover, when $\epsilon$
is small, one can show using the same lemma that the acceleration conditions
are met even for drift coefficients of order $1$ (this involves
a long and routine verification that we skip here).
We illustrate these properties in the next section.
\subsubsection{Numerical results}

In~\Cref{fig:includespectrum1,myfigure6} we consider the HJB equation~\cref{eq-pde} in one dimension $p=1$. We take the size of the discretization grid $N=500$ with $m=10$ actions at each state. We take the volatility $\sigma_1=1$ and the dissipation parameter $\lambda=1$. We generate the drift values $g_1(a,x)$ at each state $x$ and for each action $a$ randomly uniformly in the interval $[0,1]$ and we generate the rewards $r(a,x)$ randomly uniformly in $[0,100]$.
In subfigure~\ref{fig:subfig61} we display the spectrum of one matrix $P_h^\tau$. The subfigure~\ref{fig:subfig62} shows a zoom on this spectrum around the point $1$, where all the difficulty occurs. It shows that the eigenvalues of $P_h^\tau$ are not included in the peaked curve $\Gamma_{\epsilon}$ but are instead included in the more tolerant curve $\Gamma_{\epsilon}(r)$ with $r=(1-\sqrt{\epsilon}/2)/(1-\sqrt{\epsilon})$.

In~\Cref{myfigure6},
we display the performance of value iteration, accelerated policy iteration and accelerated value iteration of degree $2$.

\Cref{fig:includespectrum2,myfigure7} display the analogue plots as \Cref{fig:includespectrum1,myfigure6} with an HJB equation in dimension $p=2$, with $N=30$, $\sigma_1=\sigma_2=\sqrt{2}$, $\lambda=2$, drifts $g_1(a,x)$ in the first direction generated uniformly randomly in $[0,1]$, drifts $g_2(a,x)$ in the second direction generated uniformly randomly in $[-1,0]$ and rewards $r(a,x)$ generated randomly uniformly in $[0,100]$.

We see that for these examples the accelerated algorithms $2$A-VI and $2$A-PI converge and are faster than the classical Value Iteration algorithm. 

\textcolor{blue}{We mention though that on these two examples the classical Policy Iteration algorithm is way faster than $2$A-PI and $2$A-VI, which is expected since the size of the matrices is small, as seen in~\Cref{myfigure2,myfigure3,myfigure4,myfigure41,myfigure42,myfigure5}. However, when the size of the matrices gets bigger our iterative algorithms become 
faster than Policy Iteration like in the large scale example of~\Cref{myfigure42}.}

\begin{figure}[htbp]
  \label{fig:includespectrum1}
	\centering
	\subfigure[Spectrum of $P_h^\tau$ and the accelerable regions.]
	{
	\includegraphics[scale=0.28]{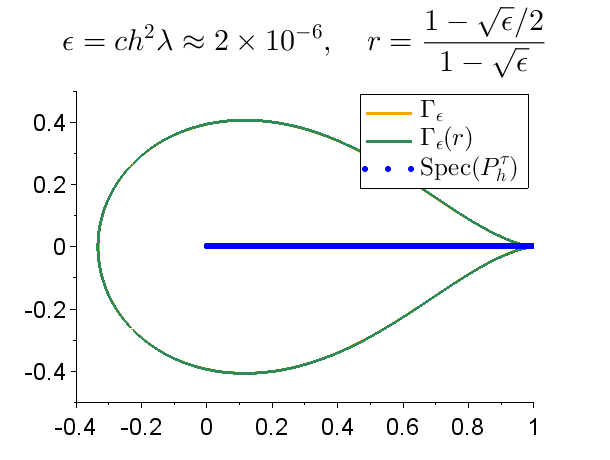}
	\label{fig:subfig61}
	}
	\subfigure[Zoom on \Cref{fig:subfig61} 
	around $1$.]
	{
	\includegraphics[scale=0.28]{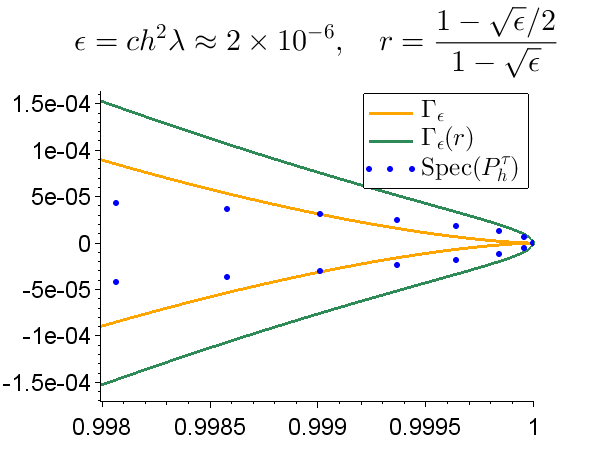}
	\label{fig:subfig62}
	}
        \caption{Spectrum of the matrix $P_h^\tau$ and acceleration region, for the HJB PDE in dimension one}
\end{figure}

\begin{figure}[htbp]
  \centering
	{
	\includegraphics[scale =0.35] {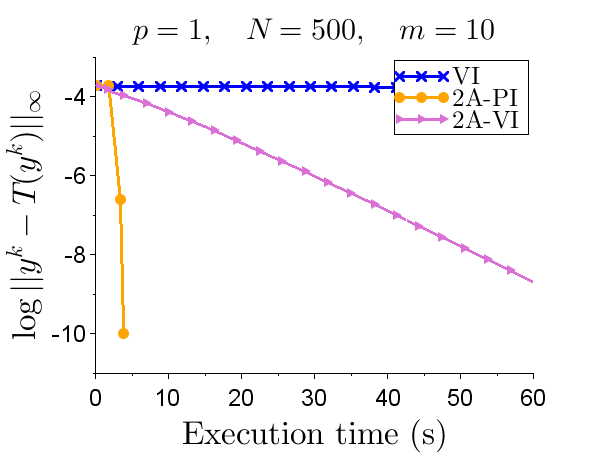}
	\label{fig:subfig63}
	}
	\caption{Solving HJB equation in one dimension with $N=500$, $\lambda=1$, $\sigma_1=1$, $g_1 \sim \mathcal{U}([0,1])$, $c=c_0/2\approx 0.5$, $\epsilon=ch^2 \lambda \approx 2\times 10^{-6}$ and $r \sim \mathcal{U}([0,100])$ .}
	\label{myfigure6}
\end{figure}


\begin{figure}[htbp]
 \label{fig:includespectrum2}
	\centering
	\subfigure[Spectrum of $P_h^\tau$ and the accelerable regions.]
	{
		\includegraphics[scale=0.28]{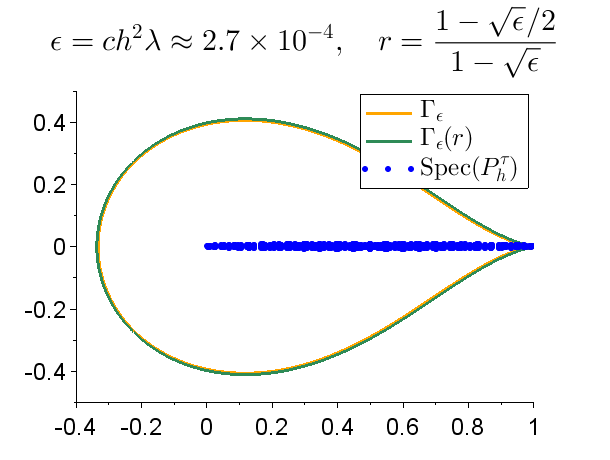}
		\label{fig:subfig71}
	}
	\subfigure[Zoom on \Cref{fig:subfig71} 
	around $1$.]
	{
		\includegraphics[scale=0.28]{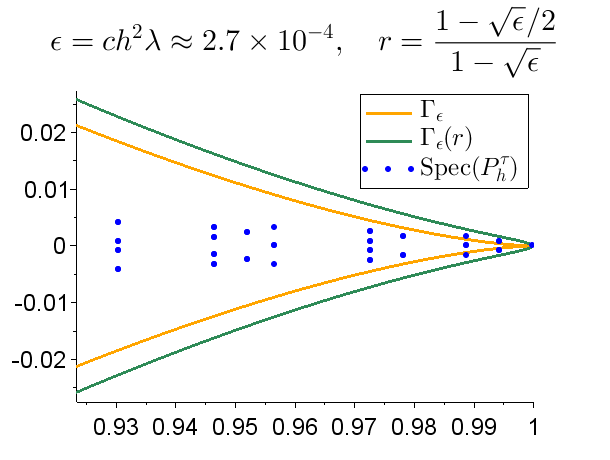}
		\label{fig:subfig72}
	}
        \caption{Spectrum of the matrix $P_h^\tau$ and acceleration region, for the HJB PDE in dimension two}
        \end{figure}
        \begin{figure}[htbp]\centering
	{
		\includegraphics[scale =0.35] {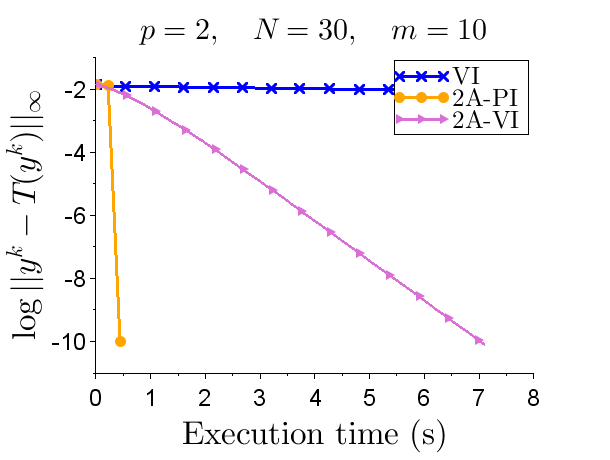}
		\label{fig:subfig73}
	}
	\caption{Solving HJB equation in two dimensions with $N=30$, $\lambda=2$, $\sigma_1=\sigma_2=\sqrt{2}$, $g_1 \sim \mathcal{U}([0,1])$, $g_2 \sim \mathcal{U}([-1,0])$, $c=c_0/2\approx 0.12$, $\epsilon=ch^2 \lambda \approx 2.7 \times 10^{-4}$ and $r \sim \mathcal{U}([0,100])$ .}
	\label{myfigure7}
\end{figure}



\section{Conclusion}
\textcolor{blue}{In this paper, we solved affine fixed point problems of type $x=g+P x$, where $P$ is a non-symmetric matrix. We showed that, if the spectrum of $P$ is contained in an explicit region of the complex plane, a Nesterov's acceleration applied to the classical value iteration algorithm does converge with an accelerated asymptotic rate of $1-\epsilon^{1/2}$, instead of the standard rate of $1-\epsilon$. Moreover, we introduced a new accelerated algorithm, of order $d\geq 2$, and showed that, under a more demanding condition on the spectrum of $P$, this algorithm converges with a multiply accelerated asymptotic rate of $1-\epsilon^{1/d}$. Using these accelerated schemes, we developed an accelerated policy iteration algorithm that solves non-linear fixed point problems arising from Markov decision processes. We illustrated the performance of the accelerated schemes on two frameworks, one using random matrices and the second 
solving an Hamilton-Jacobi-Bellman equation.
As an open problem, it remains to generalize the convergence analysis of the accelerated value iteration algorithm, of degree $d\geq 2$, to the case of non-linear fixed point problems, and in particular to Markov decision processes.}


\bibliographystyle{alpha}
\bibliography{accelerated}

\newpage

 \appendix

\end{document}